\documentclass[11pt]{article}
\usepackage[tbtags]{amsmath}
\usepackage{amssymb}
\usepackage{amsthm}
\usepackage[misc]{ifsym}
\usepackage{cases}
\usepackage{mathrsfs}
\usepackage{ulem}
\usepackage{color}
\usepackage[colorlinks,linkcolor=black,anchorcolor=black,citecolor=black]{hyperref}

\numberwithin{equation}{section}
\normalsize

\title{\bf An $\alpha$-Potential Game Approach to $N$-Player Stochastic Linear-Quadratic Differential Games\thanks{This work is supported by National Natural Science Foundations of China (12471419, 12271304), and Shandong Provincial Natural Science Foundation (ZR2024ZD35).}}

\author{\normalsize Chenhui Hao\thanks{\it School of Mathematics, Shandong University, Jinan, P.R. China, E-mail: 202411922@mail.sdu.edu.cn},\quad Jingtao Shi\thanks{\it Corresponding author. School of Mathematics, Shandong University, Jinan, P.R. China, E-mail: shijingtao@sdu.edu.cn}}

\newtheorem{mypro}{Proposition}[section]
\newtheorem{mythm}{Theorem}[section]
\newtheorem{mydef}{Definition}[section]
\newtheorem{mylem}{Lemma}[section]
\newtheorem{Remark}{Remark}[section]
\newtheorem{mycor}{Corollary}[section]
\newtheorem{assumption}{Assumption}[section]
\newtheorem{example}{Example}[section]

\begin{document}
	
	\maketitle
	
	\begin{abstract}
	This paper studies $N$-player stochastic linear-quadratic (LQ) differential games from the perspective of $\alpha$-potential games. We first consider a closed-loop LQ game with multiplicative noise, where both the drift and the diffusion coefficients depend linearly on the state and the full control vector. For this model, we derive probabilistic and partial differential equation (PDE) representations for the first- and second-order linear derivatives of the players' cost function and prove the equivalence between them. We then develop an open-loop stochastic LQ \(\alpha\)-potential game framework. Using the linear derivative construction, we build an \(\alpha\)-potential function and derive an explicit upper bound for the approximation parameter \(\alpha\) in terms of the model coefficients and the admissible control radius. Moreover, the minimization of the \(\alpha\)-potential function is reduced to a finite-dimensional stochastic control problem by augmenting the state with the variational process, which yields an open-loop \(\alpha\)-Nash equilibrium. As an application, we revisit a network LQ game considered in \cite{GuoLiZhang2025} and show that the feedback representation obtained from our approach coincides with the feedback in the existing conditional McKean--Vlasov approach, while our characterization follows directly from a standard finite-dimensional LQ control problem. 
    \end{abstract}
	
	
	\medskip
	
\section{Introduction}
	
	Stochastic differential games provide a classical framework for studying strategic interactions among multiple decision makers in continuous time. Since the foundational works of von Neumann and Morgenstern \cite{VonNeumannMorgenstern1944} and Nash \cite{Nash1950}, Nash equilibrium has become one of the central concepts in noncooperative game theory, and has been widely used in economics, engineering and management \cite{CachonZipkin1999}. However, for \(N\)-player stochastic differential games, the computation of Nash equilibria is generally difficult even in the {\it linear-quadratic} (LQ) setting.
	
	Potential games offer a different route to equilibrium analysis. Introduced by Monderer and Shapley \cite{MondererShapley1996}, potential games are games in which every player's unilateral cost variation can be represented by the variation of a single potential function. This structure reduces the search for Nash equilibrium to the optimization of only one function. The potential game paradigm has found applications in many areas involving wireless networks \cite{Yamamoto2015}, power allocation \cite{CanalesGallego2010} and offshore wind energy \cite{TaoFeijooLorenzo2024}. Recently, Guo and Zhang \cite{GuoZhang2025} developed a dynamic potential game framework by introducing a notion of linear derivatives for cost functions with respect to players' strategies. This derivative notion allows them to characterize when a dynamic game admits an exact potential function through the symmetry of mixed second-order linear derivatives.
	
	However, exact potential structures are often restrictive in dynamic games. In heterogeneous stochastic differential games, players may have different cost functions and their controls are coupled through the state dynamics, making it difficult to represent all unilateral deviations by a single exact potential function.
	
	The recently developed \(\alpha\)-potential framework relaxes the exact potential structure by allowing an error \(\alpha\) between unilateral cost variations and variations of a single \(\alpha\)-potential function. This framework has been developed for general dynamic games in \cite{GuoLiZhang2025}; it shows that minimizing an \(\alpha\)-potential function yields an \(\alpha\)-Nash equilibrium and that the parameter \(\alpha\) can be estimated through the asymmetry of mixed second-order linear derivatives. \cite{GuoLiZhangBSDE2025} develop rigorous estimates for \(\alpha\) via a {\it backward stochastic differential equation} (BSDE) approach, while \cite{GuoLiZhangJump2025} applies this framework to distributed games with jumps. Recent developments also include Markov \(\alpha\)-potential games \cite{GuoLiMaheshwariSastryWu2026}, applications to decentralized control of connected and automated vehicles \cite{DiHuWangZhang2025}, and a limit theory connecting finite-player \(\alpha\)-potential games with mean field potential games \cite{GuoWangZhang2026}.

	Although the \(\alpha\)-potential framework provides a useful tool for studying approximate Nash equilibria, its application to stochastic LQ differential games remains limited. Recent work has studied independent policy-gradient learning for stochastic LQ games under distributed policies \cite{PlankZhang2026}. In that setting, each player's policy depends only on its own state, which is exploited in the convergence analysis of the learning algorithm. Our focus is different. We consider stochastic LQ games in which the state dynamics may depend on the full control vector and may contain controlled diffusion terms. Thus the game is not restricted to the distributed policy structure.
	
	\begin{itemize}
	\item First, we extend the closed-loop LQ framework beyond the additive-noise setting considered in \cite{GuoZhang2025}. More precisely, we consider stochastic LQ games with multiplicative noise. In this setting, we establish probabilistic and {\it partial differential equation} (PDE) representations for the first- and second-order linear derivatives. We further prove the equivalence between these probabilistic and PDE representations. 
		
	\item Second, we develop an open-loop stochastic LQ $\alpha$-potential game framework. Using the linear-derivative construction of $\alpha$-potential functions, we construct an $\alpha$-potential function for the open-loop LQ game and derive an explicit global upper bound for the approximation parameter $\alpha$. We then study the minimization of the $\alpha$-potential function. In the open-loop LQ setting, inspired by \cite{GuoLiZhangJump2025} for distributed game, we exploit the linear identity $X^{ru}=X^u-(1-r)\bar Y^u$ to rewrite the same potential minimization problem as a standard finite-dimensional LQ control problem. This enables us to apply the classical verification theorem to minimize the potential function, and hence to obtain an $\alpha$-Nash equilibrium.
		
	\item Third, we apply the minimization method of open-loop \(\alpha\)-potential function to the network LQ example studied in \cite{GuoLiZhang2025}. This example illustrates how the preceding construction can simplify the analysis of a concrete LQ game. In \cite{GuoLiZhang2025}, the minimizer is described by introducing an auxiliary uniform random variable \(\tau\) and using conditional expectations involving \(X^{\tau,u}\). By contrast, we solve the same minimization problem directly as a finite-dimensional LQ control problem. The control obtained from our approach is shown to coincide with the control obtained in \cite{GuoLiZhang2025}. Thus, for this LQ example, our approach gives a more direct finite-dimensional characterization of the minimizer.
	\end{itemize}	
	
    The \(\alpha\)-potential function constructed through linear derivatives involves variational processes of the state with respect to policy perturbations. Hence its
minimization naturally requires an augmented state consisting of the original state and these variational processes, and the resulting minimizer is a feedback of the augmented state rather than of the original state alone. Since this does not fit the  closed-loop strategy class, we perform the minimization analysis in the open-loop LQ setting.
 
	The rest of the paper is organized as follows. Section 2 introduces the preliminaries on potential games, $\alpha$-potential games and linear derivatives. Section 3 studies the linear derivative of closed-loop stochastic LQ game with multiplicative noise. Section 4 develops the open-loop LQ $\alpha$-potential framework and its minimization. Section 5 applies this method to the network LQ example in \cite{GuoLiZhang2025}. Section 6 gives some concluding remarks.
	
\subsection*{Notation used in this paper}
	
	Let \(E\) be a Euclidean space with its norm denoted by $|\cdot|$. For \(1\le p<\infty\), \(L^p(0,T;E)\) denotes the space of \(p\)-th power Lebesgue integrable \(E\)-valued functions, and \(L^\infty(0,T;E)\) denotes the space of essentially bounded measurable \(E\)-valued functions. We write \(L^2(\Omega;E)\) for the space of square-integrable \(E\)-valued random variables. For \(p\ge1\),
	\begin{align*}
		\mathcal S^p([t,T];E)&:=\bigg\{X:\ X\text{ is progressively measurable and }\mathbb E\bigg[\sup_{s\in[t,T]}|X_s|^p\bigg]<\infty\bigg\},\\
		\mathcal H^p([t,T];E)&:=\left\{X:\ X\text{ is progressively measurable and }\mathbb E\left[\int_t^T|X_s|^pds\right]<\infty\right\}.
	\end{align*}
We denote by \(\mathbb S^n\) the set of \(n\times n\) symmetric matrices.
	
\section{Preliminaries}\label{sec:preliminaries}

	We first recall the analytical framework of dynamic  \(\alpha\)-potential games developed in \cite{GuoLiZhang2025,GuoZhang2025}, together with several relevant results that will be used throughout the paper.
	
\subsection{Open-loop and closed-loop strategy classes}
	
	Let \(S\) be the state space and $G=(\mathcal I_N,S,(\mathcal A_i)_{i\in\mathcal I_N},(J_i)_{i\in\mathcal I_N})$ be a general $N$-player game, where $\mathcal I_N\equiv\{1,\ldots,N\}$, $\mathcal A^{(N)}:=\prod_{i=1}^N\mathcal A_i$. Here \(\mathcal A_i\) is the set of admissible controls of player \(i\), and \(J_i:\mathcal A^{(N)}\to\mathbb R\) is player \(i\)'s cost function. For \(a\equiv(a_i)_{i\in\mathcal I_N}\in\mathcal A^{(N)}\), we write \(a_{-i}:=(a_j)_{j\ne i}\in\mathcal A_{-i}^{(N)}\), where $\mathcal A_{-i}^{(N)}:=\prod_{j\ne i}\mathcal A_j$. For a fixed initial state $s_0\in S$, a strategy profile $a^*\in\mathcal A^{(N)}$ is an open-loop Nash equilibrium if the sets $\mathcal A_i$ are open-loop control classes and
	\begin{equation}\label{eq:ol-nash}
		J_i(a_i^*,a_{-i}^*)\leq J_i(a_i,a_{-i}^*), \quad \forall i\in\mathcal I_N,\ a_i\in\mathcal A_i.
	\end{equation}
It is an $\varepsilon$-open-loop Nash equilibrium if the right-hand side in \eqref{eq:ol-nash} is replaced by $J_i(a_i,a_{-i}^*)+\varepsilon$.
	
	For closed-loop games, let $G=(\mathcal I_N,S,(A_i)_{i\in\mathcal I_N},\pi^{(N)},(V_i)_{i\in\mathcal I_N}),\pi^{(N)}:=\prod_{i=1}^N\pi_i$, where $A_i$ is the action set of player $i$, $\pi_i$ is the set of Borel measurable functions $\phi_i:S\to A_i$ and $V_i:S\times\pi^{(N)}\to\mathbb R$ is player $i$'s cost function. For $s\in S$, write $V_i^s(\phi):=V_i(s,\phi)$. A policy profile $\phi^*\in\pi^{(N)}$ is a closed-loop Nash equilibrium with initial state $s_0$ if
	\begin{equation*}
		V_i^{s_0}(\phi_i^*,\phi_{-i}^*)\leq V_i^{s_0}(\phi_i,\phi_{-i}^*),\quad \forall i\in\mathcal I_N,\ \phi_i\in\pi_i.
	\end{equation*}
	
\subsection{$\alpha$-potential games}

	\begin{mydef}[$\alpha$-potential game]
		The game $G=(\mathcal I_N,S,(\mathcal A_i)_{i\in\mathcal I_N},(J_i)_{i\in\mathcal I_N})$ is an \(\alpha\)-potential game if there exist
		\(\alpha\geq0\) and \(\Phi:\mathcal A^{(N)}\to\mathbb R\) such that
		\begin{equation*}
			\left|J_i(a_i',a_{-i})-J_i(a_i,a_{-i})-\big[\Phi(a_i',a_{-i})-\Phi(a_i,a_{-i})\big]\right|\leq \alpha,
		\end{equation*}
		for all \(i\in\mathcal I_N\), \(a_i,a_i'\in\mathcal A_i\) and \(a_{-i}\in\mathcal A_{-i}^{(N)}:=\prod_{j\ne i}\mathcal A_j\). The function
		\(\Phi\) is called an open-loop \(\alpha\)-potential function. The same definition extends to closed-loop policies of game $G=(\mathcal I_N,S,(A_i)_{i\in\mathcal I_N},
		\pi^{(N)},(V_i)_{i\in\mathcal I_N})$ by replacing \(\mathcal A_i\) with \(\pi_i\) and replacing \(J_i\) with \(V_i^{s_0}\) for a fixed initial state \(s_0\in S\), and 
        we call $ G$ a closed-loop $\alpha$-potential game.
	\end{mydef}
	
	In the special case \(\alpha=0\), we have $J_i(a_i',a_{-i})-J_i(a_i,a_{-i})=\Phi(a_i',a_{-i})-\Phi(a_i,a_{-i})$, and the game is called a potential game. The next proposition explains the main advantage of this framework: once an \(\alpha\)-potential function is available, minimizing it yields an approximate Nash equilibrium, with the approximation error explicitly controlled by \(\alpha\).
	
	\begin{mypro}
	    If \(G\) is an open-loop \(\alpha\)-potential game with \(\alpha\)-potential function \(\Phi\), and \(a^\varepsilon\in\mathcal A^{(N)}\) satisfies
	    \[
	    \Phi(a^\varepsilon)\leq \inf_{a\in\mathcal A^{(N)}}\Phi(a)+\varepsilon,
	    \]
	    then \(a^\varepsilon\) is an \((\alpha+\varepsilon)\)-open-loop Nash equilibrium. The same conclusion holds for closed-loop games by replacing
	    \(\mathcal A^{(N)}\) with \(\pi^{(N)}\).
	\end{mypro}
	
\subsection{Linear derivatives}

	To construct an $\alpha$-potential function for a given game, \cite{GuoZhang2025} introduce a new notion of derivative with respect to policies. For each \(i\in\mathcal I_N\), denote by \(\operatorname{span}(\mathcal A_i)\) the vector space of all finite linear combinations of elements in \(\mathcal A_i\):
	\[
	\operatorname{span}(\mathcal A_i):=\left\{\sum_{\ell=1}^m c_\ell a_i^{(\ell)}\ \middle|\ m\in\mathbb N,\ c_\ell\in\mathbb R,\ a_i^{(\ell)}\in\mathcal A_i\right\}.
	\]
	
	\begin{mydef}[Linear derivatives]
		Let \(\mathcal A^{(N)}=\prod_{i=1}^N\mathcal A_i\) be convex and let \(f:\mathcal A^{(N)}\to\mathbb R\). We say that \(f\) has a linear derivative with
		respect to \(\mathcal A_i\) if there exists $\frac{\delta f}{\delta a_i}:\mathcal A^{(N)}\times\operatorname{span}(\mathcal A_i)\to\mathbb R$
		such that \(\frac{\delta f}{\delta a_i}(a;\cdot)\) is linear and
		\begin{equation*}
			\lim_{\epsilon\downarrow0}\frac{f(a_i+\epsilon(a_i'-a_i),a_{-i})-f(a_i,a_{-i})}{\epsilon}=\frac{\delta f}{\delta a_i}(a;a_i'-a_i)
		\end{equation*}
		for all \(a=(a_i,a_{-i})\in\mathcal A^{(N)}\) and \(a_i'\in\mathcal A_i\).
		
		For \(i,j\in\mathcal I_N\), we say that \(f\) has second-order linear derivatives with respect to \(\mathcal A_i\times\mathcal A_j\) if \(f\) has first-order linear
		derivatives with respect to both \(\mathcal A_i\) and \(\mathcal A_j\), and if there exists $\frac{\delta^2 f}{\delta a_i\delta a_j}:
		\mathcal A^{(N)}\times\operatorname{span}(\mathcal A_i)\times\operatorname{span}(\mathcal A_j)\to\mathbb R$
		such that \(\frac{\delta^2 f}{\delta a_i\delta a_j}(a;\cdot,\cdot)\) is bilinear and, for each fixed \(a_i'\in\operatorname{span}(\mathcal A_i)\),
		\(\frac{\delta^2 f}{\delta a_i\delta a_j}(\cdot;a_i',\cdot)\) is a linear derivative of \(\frac{\delta f}{\delta a_i}(\cdot;a_i')\) with respect to \(\mathcal A_j\).
	\end{mydef}
	
	The following theorem constructs an \(\alpha\)-potential function by using linear derivatives and quantifies the approximation parameter \(\alpha\) in terms of the discrepancy between the mixed second-order linear derivatives.
	
	\begin{mythm}[{\cite{GuoLiZhang2025,GuoLiZhangJump2025}}]
		Let \(\mathcal A^{(N)}\) be convex. Suppose that, for all \(i,j\in\mathcal I_N\),\(J_i\) has second-order linear derivatives with respect to 
        \(\mathcal A_i\times\mathcal A_j\). Assume that for all \(z,a\in\mathcal A^{(N)}\), \(a_i',\widetilde a_i'\in\mathcal A_i\), and
		\(a_j''\in\mathcal A_j\), the following two conditions hold:
		\begin{align*}
			&\sup_{r,\epsilon\in[0,1]}\left|\frac{\delta^2J_i}{\delta a_i\delta a_j}\big(z+r(a^\epsilon-z);a_i',a_j''\big)\right|<\infty,\\
			&[0,1]^N\ni\epsilon\longmapsto\frac{\delta^2J_i}{\delta a_i\delta a_j}\big(z+\epsilon\cdot(a-z);a_i',a_j''\big)\quad\text{is continuous at }0.
		\end{align*}
		where
		\[
		a^\epsilon:=(a_i+\epsilon(\widetilde a_i'-a_i),a_{-i}),\qquad z+\epsilon\cdot(a-z):=(z_i+\epsilon_i(a_i-z_i))_{i=1}^N.
		\]
		Fix \(z\in\mathcal A^{(N)}\) and define
		\begin{equation}\label{eq:potential-linear-derivative}
			\Phi(a)=\int_0^1\sum_{j=1}^N\frac{\delta J_j}{\delta a_j}\big(z+r(a-z);a_j-z_j\big)\,dr .
		\end{equation}
		Then \(\Phi\) is an \(\alpha\)-potential function and
		\begin{equation*}
			\alpha\leq2\sup_{i\in\mathcal I_N,\ a_i'\in\mathcal A_i,\ a,a''\in\mathcal A^{(N)}}\sum_{j=1}^N
			\left|\frac{\delta^2 J_i}{\delta a_i\delta a_j}(a;a_i',a_j'')-\frac{\delta^2 J_j}{\delta a_j\delta a_i}(a;a_j'',a_i')\right|.
		\end{equation*}
		In particular, when $0 \in \mathcal A^{(N)}$, a sharper bound is available
		\begin{equation}\label{eq:alpha-estimate}
			\alpha\leq\sup_{i\in\mathcal I_N,\ a_i'\in\mathcal A_i,\ a,a''\in\mathcal A^{(N)}}\sum_{j=1}^N
			\left|\frac{\delta^2 J_i}{\delta a_i\delta a_j}(a;a_i',a_j'')-\frac{\delta^2 J_j}{\delta a_j\delta a_i}(a;a_j'',a_i')\right|.
		\end{equation}
	\end{mythm}
	
	If the mixed second-order linear derivatives of the cost functions are symmetric, i.e.
	\begin{equation*}
		\frac{\delta^2 J_i}{\delta a_i\delta a_j}(a;a_i',a_j'')=\frac{\delta^2 J_j}{\delta a_j\delta a_i}(a;a_j'',a_i'),
	\end{equation*}
	for all \(i,j\in\mathcal I_N\), then \(\alpha=0\), and \(\Phi\) in \eqref{eq:potential-linear-derivative} is a potential function.
	For closed-loop games, the same statement is applied with \(\mathcal A_i=\pi_i\) and \(J_i=V_i^{s_0}\). 
	
\section{Representations of linear derivatives for closed-loop stochastic LQ games}\label{sec:problem-formulation}
	
	We now specialize the closed-loop game introduced in Section~\ref{sec:preliminaries} to an LQ stochastic differential game. Throughout this section, we work on a complete filtered probability space $(\Omega,\mathcal F,\mathbb F,\mathbb P),$ where $\mathbb F\equiv(\mathcal F_s)_{s\in[0,T]}$ satisfies the usual conditions and supports a one-dimensional Brownian motion \(W\equiv(W_s)_{s\in[0,T]}\). The multidimensional case can be treated in the same way. Consider $G=\bigl(\mathcal I_N,S,(A_i)_{i\in\mathcal I_N},\pi^{(N)},(V_i)_{i\in\mathcal I_N}\bigr)$, where \(\mathcal I_N\equiv\{1,\ldots,N\}\), \(S=[0,T]\times\mathbb R^n\), \(A_i=\mathbb R^{k_i}\), \(\pi^{(N)}:=\prod_{i=1}^N\pi_i\) is the set of closed-loop policy profiles and \(k:=\sum_{i=1}^Nk_i\). Let $E_i\in\mathbb R^{k\times k_i}$ be the block matrix whose $i$-th block is the $k_i\times k_i$ identity matrix, corresponding to player $i$'s action, and whose other blocks are zero. For each player \(i\in\mathcal I_N\), define
	\[
    \mathcal K_i:=L^2(0,T;\mathbb R^{k_i\times n}),\qquad\mathcal M_i:=L^2(0,T;\mathbb R^{k_i}),
    \]
and let \(\pi_i\) be the set of feedback policies of the form
    \[
    \phi_i(t,x)=K_i(t)x+m_i(t),\qquad t\in[0,T],\qquad (K_i,m_i)\in\mathcal K_i\times\mathcal M_i.
    \] 
With an abuse of notation, we identify \(\phi_i\in\pi_i\) with the pair \((K_i,m_i)\), and identify \(\phi=(\phi_i)_{i\in\mathcal I_N}\in\pi^{(N)}\) with
\((K,m)=((K_i)_{i\in\mathcal I_N},(m_i)_{i\in\mathcal I_N})\). We write
	\[
    K_s:=\begin{pmatrix}K_1(s)\\ \vdots\\ K_N(s)\end{pmatrix},\qquad m_s:=\begin{pmatrix}m_1(s)\\ \vdots\\ m_N(s)\end{pmatrix}.
    \]
For each \((t,x)\in S\) and \(\phi=(K,m)\in\pi^{(N)}\), let \(X^{t,x,\phi}\) satisfy the closed-loop state equation
	\begin{equation}\label{eq:closed-state}
		\left\{
		\begin{aligned}
			dX_s^{t,x,\phi}={}&\Big((A_s+B_sK_s)X_s^{t,x,\phi}+B_sm_s+b_s\Big)\,ds \\
			&+\Big((C_s+D_sK_s)X_s^{t,x,\phi}+D_sm_s+\sigma_s\Big)\,dW_s,\qquad s\in[t,T],\\
			X_t^{t,x,\phi}={}&x.
		\end{aligned}
		\right.
	\end{equation}
	Define player \(i\)'s cost function \(V_i:S\times\pi^{(N)}\to\mathbb R\) by
	\begin{equation}\label{eq:LQ-cost}
		\begin{aligned}
			V_i^{t,x}(\phi)=\mathbb E\bigg[&\int_t^T \frac12\Big( (X_s^{t,x,\phi})^\top Q_i(s)X_s^{t,x,\phi}
            +\big(K_sX_s^{t,x,\phi}+m_s\big)^\top R_i(s)\big(K_sX_s^{t,x,\phi}+m_s\big) \\
			&+2(X_s^{t,x,\phi})^\top S_i(s)\big(K_sX_s^{t,x,\phi}+m_s\big)+2q_i(s)^\top X_s^{t,x,\phi}\\
			&+2\rho_i(s)^\top\big(K_sX_s^{t,x,\phi}+m_s\big)\Big)ds +\frac12\Big((X_T^{t,x,\phi})^\top G_iX_T^{t,x,\phi}+2g_i^\top X_T^{t,x,\phi}\Big)\bigg].
		\end{aligned}
	\end{equation}
	We impose the following standard conditions on the coefficients of \eqref{eq:closed-state} and \eqref{eq:LQ-cost}.
	\begin{assumption}\label{ass:LQ-data}
		The state coefficients satisfy
		\[
        A,C\in L^\infty(0,T;\mathbb R^{n\times n}),\qquad B,D\in L^\infty(0,T;\mathbb R^{n\times k}),\qquad b,\sigma\in L^2(0,T;\mathbb R^n).
        \]
		For each \(i\in\mathcal I_N\), the cost coefficients satisfy
		\[
        Q_i\in L^\infty(0,T;\mathbb S^n),\qquad R_i\in L^\infty(0,T;\mathbb S^k),\qquad S_i\in L^\infty(0,T;\mathbb R^{n\times k}),
        \]
		\[
        q_i\in L^2(0,T;\mathbb R^n),\qquad\rho_i\in L^2(0,T;\mathbb R^k),\qquad G_i\in\mathbb S^n,\qquad g_i\in\mathbb R^n.
        \]
	\end{assumption}
	
	Set
	\begin{equation*}
		\bar A:=A+BK,\qquad\bar b:=b+Bm,\qquad\bar C:=C+DK,\qquad\bar\sigma:=\sigma+Dm.
	\end{equation*}
	By Assumption~\ref{ass:LQ-data} and \((K,m)\in\prod_{i=1}^N(\mathcal K_i\times\mathcal M_i)\),
	\[
    \bar A,\bar C\in L^2(0,T;\mathbb R^{n\times n}),\qquad\bar b,\bar\sigma\in L^2(0,T;\mathbb R^n).
    \]
	Hence, for every\((t,x)\) and \(\phi\in\pi^{(N)}\), \eqref{eq:closed-state} admits a unique strong solution in \(\mathcal S^2([t,T];\mathbb R^n)\), and \eqref{eq:LQ-cost} is well defined.
	
	We write
	\[
    B_i(t):=B(t)E_i,\qquad D_i(t):=D(t)E_i.
    \]

\subsection{Probabilistic representations of linear derivatives}\label{sec:probabilistic-representations}
	
    In this section, we derive the probabilistic representations in the LQ setting. Fix a policy profile \(\phi=(K,m)\in\pi^{(N)}\), and let \(X\equiv X^{t,x,\phi}\) be the
solution of
	\[
	\left\{
	\begin{aligned}
		dX_s&=(\bar A_sX_s+\bar b_s)\,ds+(\bar C_sX_s+\bar\sigma_s)\,dW_s,\\
		X_t&=x,
	\end{aligned}
	\right.
	\]
where \(\bar A=A+BK\), \(\bar b=b+Bm\), \(\bar C=C+DK\), and \(\bar\sigma=\sigma+Dm\). 
	
	Fix \(h\in\mathcal I_N\) and take a direction
	\[
	\phi_h'(s,x)=K_h'(s)x+m_h'(s),\qquad (K_h',m_h')\in\mathcal K_h\times\mathcal M_h.
	\]
	Recall that \(B_h:=BE_h\) and \(D_h:=DE_h\). The first-order variational process $Y_s^h$ satisfies
	\begin{equation*}
		\left\{
		\begin{aligned}
			dY_s^h={}&\left[\bar A_sY_s^h+B_h(s)(K_h'(s)X_s+m_h'(s))\right]ds\\
			&+\left[\bar C_sY_s^h+D_h(s)(K_h'(s)X_s+m_h'(s))\right]dW_s,\\
			Y_t^h={}&0.
		\end{aligned}
		\right.
	\end{equation*}
    Similarly to \cite{GuoZhang2025}, the standard moment estimates for linear {\it stochastic differential equations} (SDEs) yield the following representation.	
	\begin{mypro}[First-order probabilistic representation]\label{First-order probabilistic representation}
		In the LQ setting of Section~\ref{sec:problem-formulation}, for every \(i,h\in\mathcal I_N\), we have
		\begin{equation}\label{eq:first-prob-explicit}
			\begin{aligned}
				\frac{\delta V_i^{t,x}}{\delta\phi_h}(\phi;\phi_h')={}&\mathbb E\bigg[\int_t^T\Big\{\big[Q_i(s)X_s+S_i(s)(K_sX_s+m_s)+q_i(s)\big]^\top Y_s^h\\
				&\qquad+\big[R_i(s)(K_sX_s+m_s)+S_i(s)^\top X_s+\rho_i(s)\big]^\top\big[K_sY_s^h\\
				&\qquad+E_h(K_h'(s)X_s+m_h'(s))\big]\Big\}\,ds+(G_iX_T+g_i)^\top Y_T^h\bigg].
			\end{aligned}
		\end{equation}
	\end{mypro}
	
	Let \(\ell\in\mathcal I_N\) and take another direction
	\[
	\phi_\ell''(s,x)=K_\ell''(s)x+m_\ell''(s),\qquad (K_\ell'',m_\ell'')\in\mathcal K_\ell\times\mathcal M_\ell.
	\]
	The first-order variational process $Y_s^\ell$ satisfies
	\begin{equation*}
		\left\{
		\begin{aligned}
			dY_s^\ell={}&\left[\bar A_sY_s^\ell+B_\ell(s)(K_\ell''(s)X_s+m_\ell''(s))\right]ds\\
			&+\left[\bar C_sY_s^\ell+D_\ell(s)(K_\ell''(s)X_s+m_\ell''(s))\right]dW_s,\\
			Y_t^\ell={}&0,
		\end{aligned}
		\right.
	\end{equation*}
	The mixed second-order variational process $Z_s^{h,\ell}$ satisfies
	\begin{equation*}
		\left\{
		\begin{aligned}
			dZ_s^{h,\ell}={}&\left[\bar A_sZ_s^{h,\ell}+B_h(s)K_h'(s)Y_s^\ell+B_\ell(s)K_\ell''(s)Y_s^h\right]ds\\
			&+\left[\bar C_sZ_s^{h,\ell}+D_h(s)K_h'(s)Y_s^\ell+D_\ell(s)K_\ell''(s)Y_s^h\right]dW_s,\\
			Z_t^{h,\ell}={}&0.
		\end{aligned}
		\right.
	\end{equation*}
    Similarly to Proposition~\ref{First-order probabilistic representation}, we have the following result.
	
	\begin{mypro}[Second-order probabilistic representation]\label{Second-order probabilistic representation}
		In the LQ setting of Section~\ref{sec:problem-formulation}, for every \(i,h,\ell\in\mathcal I_N\),
		\begin{equation}\label{eq:second-prob-explicit}
			\begin{aligned}
				&\frac{\delta^2 V_i^{t,x}}{\delta\phi_h\delta\phi_\ell}(\phi;\phi_h',\phi_\ell'')
                ={}\mathbb E\bigg[\int_t^T\bigg\{(Y_s^\ell)^\top Q_i(s)Y_s^h\\
				&\qquad+\big[K_sY_s^\ell+E_\ell(K_\ell''(s)X_s+m_\ell''(s))\big]^\top R_i(s)\big[K_sY_s^h+E_h(K_h'(s)X_s+m_h'(s))\big]\\
				&\qquad+(Y_s^\ell)^\top S_i(s)\big[K_sY_s^h+E_h(K_h'(s)X_s+m_h'(s))\big]\\
				&\qquad+(Y_s^h)^\top S_i(s)\big[K_sY_s^\ell+E_\ell(K_\ell''(s)X_s+m_\ell''(s))\big]\\
				&\qquad+\big[Q_i(s)X_s+S_i(s)(K_sX_s+m_s)+q_i(s)\big]^\top Z_s^{h,\ell}\\
				&\qquad+\big[R_i(s)(K_sX_s+m_s)+S_i(s)^\top X_s+\rho_i(s)\big]^\top\big[K_sZ_s^{h,\ell}\\
				&\qquad+E_hK_h'(s)Y_s^\ell+E_\ell K_\ell''(s)Y_s^h\big]\bigg\}\,ds+(Y_T^\ell)^\top G_iY_T^h+(G_iX_T+g_i)^\top Z_T^{h,\ell}\bigg].
			\end{aligned}
		\end{equation}
	\end{mypro}
	
\subsection{PDE representations in the LQ setting}\label{sec:pde-representations}

    For the PDE representations, define the running cost
    \begin{equation*}
       f_i(t,x,a):=\frac12\Big(x^\top Q_i(t)x+a^\top R_i(t)a+2x^\top S_i(t)a+2q_i(t)^\top x+2\rho_i(t)^\top a\Big).
    \end{equation*}
    For \(y\in\mathbb R^n\), \(z\in\mathbb S^n\), and \(a\in\mathbb R^k\), define
    \begin{equation*}
    \begin{aligned}
        \mathcal L(t,x,y,z,a):={}&(A(t)x+B(t)a+b(t))^\top y\\
        &+\frac12\operatorname{Tr}\!\left((C(t)x+D(t)a+\sigma(t))(C(t)x+D(t)a+\sigma(t))^\top z\right),
    \end{aligned}
    \end{equation*}
and the Hamiltonian
    \begin{equation*}
        H_i(t,x,y,z,a):=\mathcal L(t,x,y,z,a)+f_i(t,x,a).
    \end{equation*}
    For a fixed feedback profile \(\phi=(K,m)\), let
    \begin{equation*}
        \mathcal L^\phi u(t,x):=\mathcal L\bigl(t,x,\partial_xu(t,x),\partial_{xx}^2u(t,x),\phi(t,x)\bigr).
    \end{equation*}
    Here \(\partial_{a_h}\) and \(\partial_{a_ha_\ell}\) denote the derivatives with respect to the corresponding action blocks. The time argument is omitted
from the coefficients when no confusion occurs.

    Set
    \[
    \bar Q_i:=Q_i+K^\top R_iK+S_iK+K^\top S_i^\top,
    \]
    \[
    \bar q_i:=q_i+S_im+K^\top R_im+K^\top\rho_i,\qquad\bar r_i:=m^\top R_im+2\rho_i^\top m.
    \]
    Define
    \begin{equation*}
    \begin{aligned}
        v_i^\phi(t,x):={}&\frac12x^\top P_i(t)x+p_i(t)^\top x+r_i(t),
    \end{aligned}
    \end{equation*}
    where \((P_i,p_i,r_i)\) is the solution of
    \begin{equation}\label{eq:P-ode}
    \left\{
    \begin{aligned}
        &\dot P_i+\bar A^\top P_i+P_i\bar A+\bar C^\top P_i\bar C+\bar Q_i=0,\\
        &P_i(T)=G_i,
    \end{aligned}
    \right.
    \end{equation}
    \begin{equation}\label{eq:p-ode}
    \left\{
    \begin{aligned}
        &\dot p_i+\bar A^\top p_i+P_i\bar b+\bar C^\top P_i\bar\sigma+\bar q_i=0,\\
        &p_i(T)=g_i,
    \end{aligned}
    \right.
    \end{equation}
    \begin{equation}\label{eq:r-ode}
    \left\{
    \begin{aligned}
        &\dot r_i+\bar b^\top p_i+\frac12\bar\sigma^\top P_i\bar\sigma+\frac12\bar r_i=0,\\
        &r_i(T)=0.
    \end{aligned}
    \right.
    \end{equation}

    A direct substitution shows that \(v_i^\phi\) satisfies, for almost every \(t\in[0,T]\) and every \(x\in\mathbb R^n\),
    \begin{equation}
    \left\{
    \begin{aligned}
        &\partial_t V(t,x)+\mathcal L^\phi V(t,x)+f_i\bigl(t,x,\phi(t,x)\bigr)=0,\\
        &V(T,x)=\frac12x^\top G_ix+g_i^\top x.
    \end{aligned}
    \right.
    \end{equation}

\subsubsection{First-order linear derivative}

    Fix \(h\in\mathcal I_N\) and a direction
    \[
    \phi_h'(t,x)=K_h'(t)x+m_h'(t).
    \]
    Set
    \begin{equation}\label{eq:Mih-definition}
        M_i^h:=B_h^\top P_i+D_h^\top P_i\bar C+E_h^\top(R_iK+S_i^\top)\in\mathbb R^{k_h\times n},
    \end{equation}
    \begin{equation}\label{eq:ellih-definition}
        \ell_i^h:=B_h^\top p_i+D_h^\top P_i\bar\sigma+E_h^\top(R_im+\rho_i)\in\mathbb R^{k_h}.
    \end{equation}
    \[
    \Theta_i^h(t,x)=(M_i^h(t)x+\ell_i^h(t))^\top(K_h'(t)x+m_h'(t)).
    \]
    Equivalently,
    \[
    \Theta_i^h(t,x)=\frac12x^\top\Theta_{2,i}^h(t)x+(\Theta_{1,i}^h(t))^\top x+\Theta_{0,i}^h(t),
    \]
where
    \begin{equation}\label{eq:Theta2ih}
        \Theta_{2,i}^h=(M_i^h)^\top K_h'+(K_h')^\top M_i^h,
    \end{equation}
    \begin{equation}\label{eq:Theta1ih}
\Theta_{1,i}^h
=(M_i^h)^\top m_h'+(K_h')^\top\ell_i^h,
\end{equation}
    \begin{equation*}
        \Theta_{0,i}^h=(\ell_i^h)^\top m_h'.
    \end{equation*}
    Define
    \begin{equation}\label{eq:w-quadratic-representation}
        w_i^h(t,x):=\frac12x^\top P_i^h(t)x+(p_i^h(t))^\top x+r_i^h(t),
    \end{equation}
where \((P_i^h,p_i^h,r_i^h)\) is the unique solution of
    \begin{equation}\label{eq:Ph-ode}
    \left\{
    \begin{aligned}
        &\dot P_i^h+\bar A^\top P_i^h+P_i^h\bar A+\bar C^\top P_i^h\bar C+\Theta_{2,i}^h=0,\\
        &P_i^h(T)=0,
    \end{aligned}
    \right.
    \end{equation}
    \begin{equation}\label{eq:ph-ode}
    \left\{
    \begin{aligned}
        &\dot p_i^h+\bar A^\top p_i^h+P_i^h\bar b+\bar C^\top P_i^h\bar\sigma+\Theta_{1,i}^h=0,\\
        &p_i^h(T)=0,
    \end{aligned}
    \right.
    \end{equation}
    \begin{equation}\label{eq:rh-ode}
    \left\{
    \begin{aligned}
        &\dot r_i^h+\bar b^\top p_i^h+\frac12\bar\sigma^\top P_i^h\bar\sigma+\Theta_{0,i}^h=0,\\
        &r_i^h(T)=0.
    \end{aligned}
    \right.
    \end{equation}

    A direct substitution shows that \(w_i^h\) satisfies, for almost every \(t\in[0,T]\) and every \(x\in\mathbb R^n\),
    \begin{equation}
    \left\{
    \begin{aligned}
        &\partial_t U(t,x)+\mathcal L^\phi U(t,x)\\
        &\quad+\partial_{a_h}H_i\!\left(t,x,\partial_x v_i^\phi(t,x),\partial_{xx}^2 v_i^\phi(t,x),\phi(t,x)\right)^\top\phi_h'(t,x)=0,\\
        &U(T,x)=0.
    \end{aligned}
    \right.
    \end{equation}

\subsubsection{Second-order linear derivative}

    Fix \(\ell\in\mathcal I_N\) and another direction
    \[
    \phi_\ell''(t,x)=K_\ell''(t)x+m_\ell''(t).
    \]
    Let \(w_i^\ell\) denote the function in \eqref{eq:w-quadratic-representation} corresponding to \(\phi_\ell''\), with coefficients \((P_i^\ell,p_i^\ell,r_i^\ell)\). Define
    \[
    N_i^{h,\ell}:=B_\ell^\top P_i^h+D_\ell^\top P_i^h\bar C\in\mathbb R^{k_\ell\times n},\qquad
    \nu_i^{h,\ell}:=B_\ell^\top p_i^h+D_\ell^\top P_i^h\bar\sigma\in\mathbb R^{k_\ell},
    \]
    \[
    N_i^{\ell,h}:=B_h^\top P_i^\ell+D_h^\top P_i^\ell\bar C\in\mathbb R^{k_h\times n},\qquad
    \nu_i^{\ell,h}:=B_h^\top p_i^\ell+D_h^\top P_i^\ell\bar\sigma\in\mathbb R^{k_h},
    \]
    \[
    H_i^{h,\ell}:=D_h^\top P_iD_\ell\in\mathbb R^{k_h\times k_\ell},\qquad
    J_i^{\ell,h}:=E_\ell^\top R_iE_h\in\mathbb R^{k_\ell\times k_h}.
    \]
    \begin{equation}\label{eq:source-second-expanded-quadratic}
    \begin{aligned}
        \mathscr R_i^{h,\ell}(t,x)={}&\big(N_i^{h,\ell}x+\nu_i^{h,\ell}\big)^\top\big(K_\ell''x+m_\ell''\big)
        +\big(N_i^{\ell,h}x+\nu_i^{\ell,h}\big)^\top\big(K_h'x+m_h'\big)\\
        &+\big(K_h'x+m_h'\big)^\top H_i^{h,\ell}\big(K_\ell''x+m_\ell''\big)+\big(K_\ell''x+m_\ell''\big)^\top J_i^{\ell,h}\big(K_h'x+m_h'\big).
    \end{aligned}
    \end{equation}
    Write
    \[
    \mathscr R_i^{h,\ell}(t,x)=\frac12x^\top\mathscr R_{2,i}^{h,\ell}(t)x+(\mathscr R_{1,i}^{h,\ell}(t))^\top x+\mathscr R_{0,i}^{h,\ell}(t),
    \]
where
    \begin{equation*}
    \begin{aligned}
        \mathscr R_{2,i}^{h,\ell}={}&(N_i^{h,\ell})^\top K_\ell''+(K_\ell'')^\top N_i^{h,\ell}+(N_i^{\ell,h})^\top K_h'+(K_h')^\top N_i^{\ell,h}\\
        &+(K_h')^\top H_i^{h,\ell}K_\ell''+(K_\ell'')^\top(H_i^{h,\ell})^\top K_h'\\
        &+(K_\ell'')^\top J_i^{\ell,h}K_h'+(K_h')^\top(J_i^{\ell,h})^\top K_\ell'',\\
        \mathscr R_{1,i}^{h,\ell}={}&(N_i^{h,\ell})^\top m_\ell''+(K_\ell'')^\top\nu_i^{h,\ell}+(N_i^{\ell,h})^\top m_h'+(K_h')^\top\nu_i^{\ell,h}\\
        &+(K_h')^\top H_i^{h,\ell}m_\ell''+(K_\ell'')^\top(H_i^{h,\ell})^\top m_h'\\
        &+(K_\ell'')^\top J_i^{\ell,h}m_h'+(K_h')^\top(J_i^{\ell,h})^\top m_\ell'',\\
        \mathscr R_{0,i}^{h,\ell}={}&(\nu_i^{h,\ell})^\top m_\ell''+(\nu_i^{\ell,h})^\top m_h'+(m_h')^\top H_i^{h,\ell}m_\ell''+(m_\ell'')^\top J_i^{\ell,h}m_h'.
    \end{aligned}
    \end{equation*}
    Define
    \begin{equation}\label{eq:z-quadratic-representation}
        z_i^{h,\ell}(t,x):=\frac12x^\top P_i^{h,\ell}(t)x+(p_i^{h,\ell}(t))^\top x+r_i^{h,\ell}(t),
    \end{equation}
where \((P_i^{h,\ell},p_i^{h,\ell},r_i^{h,\ell})\) is the unique solution of
    \begin{equation}\label{eq:Phl-ode}
    \left\{
    \begin{aligned}
        &\dot P_i^{h,\ell}+\bar A^\top P_i^{h,\ell}+P_i^{h,\ell}\bar A+\bar C^\top P_i^{h,\ell}\bar C+\mathscr R_{2,i}^{h,\ell}=0,\\
        &P_i^{h,\ell}(T)=0,
    \end{aligned}
    \right.
    \end{equation}
    \begin{equation}\label{eq:phl-ode}
    \left\{
    \begin{aligned}
        &\dot p_i^{h,\ell}+\bar A^\top p_i^{h,\ell}+P_i^{h,\ell}\bar b+\bar C^\top P_i^{h,\ell}\bar\sigma+\mathscr R_{1,i}^{h,\ell}=0,\\
        &p_i^{h,\ell}(T)=0,
    \end{aligned}
    \right.
    \end{equation}
    \begin{equation}\label{eq:rhl-ode}
    \left\{
    \begin{aligned}
        &\dot r_i^{h,\ell}+\bar b^\top p_i^{h,\ell}+\frac12\bar\sigma^\top P_i^{h,\ell}\bar\sigma+\mathscr R_{0,i}^{h,\ell}=0,\\
        &r_i^{h,\ell}(T)=0.
    \end{aligned}
    \right.
    \end{equation}
A direct substitution shows that \(z_i^{h,\ell}\) satisfies, for almost every \(t\in[0,T]\) and every \(x\in\mathbb R^n\),
    \begin{equation}
    \left\{
    \begin{aligned}
        &\partial_t W(t,x)+\mathcal L^\phi W(t,x)\\
        &\quad+\partial_{a_\ell}\mathcal L\!\left(t,x,\partial_x w_i^h(t,x),\partial_{xx}^2 w_i^h(t,x),\phi(t,x)\right)^\top\phi_\ell''(t,x)\\
        &\quad+\partial_{a_h}\mathcal L\!\left(t,x,\partial_x w_i^\ell(t,x),\partial_{xx}^2 w_i^\ell(t,x),\phi(t,x)\right)^\top\phi_h'(t,x)\\
        &\quad+\phi_h'(t,x)^\top\partial_{a_ha_\ell}H_i\!\left(t,x,\partial_x v_i^\phi(t,x),\partial_{xx}^2 v_i^\phi(t,x),\phi(t,x)\right)\phi_\ell''(t,x)=0,\\
        &W(T,x)=0.
    \end{aligned}
    \right.
    \end{equation}
Under Assumption~\ref{ass:LQ-data}, $\bar A,\bar C,\bar b,\bar\sigma\in L^2(0,T),\bar Q_i,\bar q_i,\bar r_i\in L^1(0,T)$. After solving \eqref{eq:P-ode}--\eqref{eq:r-ode}, one has
\(M_i^h,\ell_i^h\in L^2(0,T)\), and therefore \(\Theta_{2,i}^h,\Theta_{1,i}^h,\Theta_{0,i}^h\in L^1(0,T)\). Similarly, \(N_i^{h,\ell},N_i^{\ell,h},\nu_i^{h,\ell}, \nu_i^{\ell,h}\in L^2(0,T)\) and \(H_i^{h,\ell},J_i^{\ell,h}\in L^\infty(0,T)\), which imply \(\mathscr R_{2,i}^{h,\ell},\mathscr R_{1,i}^{h,\ell}, \mathscr R_{0,i}^{h,\ell}\in L^1(0,T)\). Hence all the {\it ordinary differential equation} (ODE) systems above are well posed.

\subsection{Equivalence between probabilistic and PDE representations}
	
	The next theorem proves the equivalence between the probabilistic representation in Section~\ref{sec:probabilistic-representations} and the PDE representation in Section~\ref{sec:pde-representations} for the LQ model.
	
	\begin{mythm}\label{thm:probabilistic-pde-equivalence}
		In the LQ setting of Section~\ref{sec:problem-formulation}, for all
		\[
		\phi=(K,m)\in\pi^{(N)},\qquad\phi_h'(t,x)=K_h'(t)x+m_h'(t),\qquad\phi_\ell''(t,x)=K_\ell''(t)x+m_\ell''(t),
		\]
		and all \((t,x)\in[0,T]\times\mathbb R^n\), the probabilistic and PDE representations of first and second order linear derivatives agree:
		\begin{equation}\label{eq:first-equivalence-clean}
			\frac{\delta V_i^{t,x}}{\delta\phi_h}(\phi;\phi_h')=w_i^h(t,x),
		\end{equation}
		\begin{equation}\label{eq:second-equivalence-clean}
			\frac{\delta^2 V_i^{t,x}}{\delta\phi_h\delta\phi_\ell}(\phi;\phi_h',\phi_\ell'')=z_i^{h,\ell}(t,x),
		\end{equation}
		where \(w_i^h\) and \(z_i^{h,\ell}\) are given by \eqref{eq:w-quadratic-representation} and \eqref{eq:z-quadratic-representation}, respectively.
	\end{mythm}
	
	\begin{proof}
		Fix \(i,h,\ell\in\mathcal I_N\), \(\phi=(K,m)\in\pi^{(N)}\), \(\phi_h'(s,x)=K_h'(s)x+m_h'(s)\), \(\phi_\ell''(s,x)=K_\ell''(s)x+m_\ell''(s)\), and \((t,x)\in[0,T]\times\mathbb R^n\).
		
		Applying It\^o's formula to \(w_i^h(s,X_s)\) in \eqref{eq:w-quadratic-representation}, and using \eqref{eq:Ph-ode}--\eqref{eq:rh-ode} together with \(w_i^h(T,\cdot)=0\),
		\begin{equation*}
			w_i^h(t,x)=\mathbb E\int_t^T\Theta_i^h(s,X_s)\,ds,
		\end{equation*}
		where
		\begin{equation*}
	    \begin{aligned}
			\Theta_i^h(s,X_s)={}&\Big[B_h^\top(P_iX+p_i)+D_h^\top P_i(\bar CX+\bar\sigma)+E_h^\top(R_i(KX+m)+S_i^\top X+\rho_i)\Big]^\top\phi_h'(s,X).
		\end{aligned}
		\end{equation*}
		Since
		\[
		\begin{aligned}
			d(P_iX+p_i)={}&\Big\{\dot P_iX+P_i(\bar AX+\bar b)+\dot p_i\Big\}\,ds+P_i(\bar CX+\bar\sigma)\,dW\\
            ={}&\Big\{-\bar A^\top(P_iX+p_i)-\bar C^\top P_i(\bar CX+\bar\sigma)-(\bar Q_iX+\bar q_i)\Big\}\,ds+P_i(\bar CX+\bar\sigma)\,dW,
		\end{aligned}
		\]
		where the second equality follows from \eqref{eq:P-ode}--\eqref{eq:p-ode}. Thus It\^o's formula gives
		\begin{equation}\label{eq:proof-first-product-full}
			\begin{aligned}
				&\mathbb E\big[(G_iX_T+g_i)^\top Y_T^h\big]\\
				={}&\mathbb E\int_t^T\Big\{\big[-\bar A^\top(P_iX+p_i)-\bar C^\top P_i(\bar CX+\bar\sigma)-(\bar Q_iX+\bar q_i)\big]^\top Y^h\\
				&\qquad +(P_iX+p_i)^\top(\bar AY^h+B_h\phi_h'(s,X))+(\bar CX+\bar\sigma)^\top P_i(\bar CY^h+D_h\phi_h'(s,X))\Big\}\,ds\\
                =&\mathbb E\int_t^T\Big\{-(\bar Q_iX+\bar q_i)^\top Y^h+\big[B_h^\top(P_iX+p_i)+D_h^\top P_i(\bar CX+\bar\sigma)\big]^\top\phi_h'(s,X)\Big\}\,ds.
			\end{aligned}
		\end{equation}
		Substituting \eqref{eq:proof-first-product-full} into \eqref{eq:first-prob-explicit} gives
		\begin{align*}
			\frac{\delta V_i^{t,x}}{\delta\phi_h}(\phi;\phi_h')
			&={}\mathbb E\int_t^T\Big\{(\bar Q_iX+\bar q_i)^\top Y^h+(R_i(KX+m)+S_i^\top X+\rho_i)^\top E_h\phi_h'(s,X)\Big\}\,ds\\
			&\quad +\mathbb E\big[(G_iX_T+g_i)^\top Y_T^h\big]\\
            &=\mathbb E\int_t^T\Big[B_h^\top(P_iX+p_i)+D_h^\top P_i(\bar CX+\bar\sigma)+E_h^\top(R_i(KX+m)\\
            &\qquad\qquad +S_i^\top X+\rho_i)\Big]^\top\phi_h'(s,X)\,ds.\\
			&=w_i^h(t,x).
		\end{align*}
		This proves \eqref{eq:first-equivalence-clean}. We apply It\^o's formula to
		\[
		s\longmapsto (Y_s^\ell)^\top P_i(s)Y_s^h+(P_i(s)X_s+p_i(s))^\top Z_s.
		\]
		The first product gives
		\begin{equation}\label{eq:proof-Y-P-Y-full}
			\begin{aligned}
				&d\big[(Y^\ell_s)^\top P_i(s)Y_s^h\big]  \\
				={}&(dY^\ell_s)^\top P_iY_s^h+(Y^\ell_s)^\top \dot P_iY_s^h\,ds+(Y^\ell_s)^\top P_i\,dY_s^h+(dY^\ell_s)^\top P_i\,dY_s^h  \\
				={}&\Big\{(Y^\ell_s)^\top(\dot P_i+\bar A^\top P_i+P_i\bar A+\bar C^\top P_i\bar C)Y_s^h  \\
				&\quad +\big[B_\ell^\top P_iY_s^h+D_\ell^\top P_i\bar CY_s^h\big]^\top\phi_\ell''(s,X_s)
                +\big[B_h^\top P_iY^\ell_s+D_h^\top P_i\bar CY^\ell_s\big]^\top\phi_h'(s,X_s) \\
				&\quad +\phi_\ell''(s,X_s)^\top D_\ell^\top P_iD_h\phi_h'(s,X_s)\Big\}\,ds+dM_s^{(1)}  \\
				={}&\Big\{-(Y^\ell_s)^\top\bar Q_iY_s^h+\big[B_\ell^\top P_iY_s^h+D_\ell^\top P_i\bar CY_s^h\big]^\top\phi_\ell''(s,X_s) \\
				&\quad+\big[B_h^\top P_iY^\ell_s+D_h^\top P_i\bar CY^\ell_s\big]^\top\phi_h'(s,X_s)+\phi_\ell''(s,X_s)^\top D_\ell^\top P_iD_h\phi_h'(s,X_s)\Big\}\,ds
                +dM_s^{(1)}.
			\end{aligned}
		\end{equation}
		where \(M^{(1)}\) is a martingale and the last equality uses \eqref{eq:P-ode}. For the second product,
		\begin{equation}\label{eq:proof-PX-Z-full}
			\begin{aligned}
				&d\big[(P_iX+p_i)^\top Z\big]=\Big\{\big[-\bar A^\top(P_iX+p_i)-\bar C^\top P_i(\bar CX+\bar\sigma)-(\bar Q_iX+\bar q_i)\big]^\top Z\\
				&\quad +(P_iX+p_i)^\top\big(\bar AZ+B_hK_h'Y^\ell+B_\ell K_\ell''Y^h\big)\\
				&\quad +(\bar CX+\bar\sigma)^\top P_i\big(\bar CZ+D_hK_h'Y^\ell+D_\ell K_\ell''Y^h\big)\Big\}\,ds+dM_s^{(2)}\\
				={}&\Big\{-(\bar Q_iX+\bar q_i)^\top Z+\big[B_h^\top(P_iX+p_i)+D_h^\top P_i(\bar CX+\bar\sigma)\big]^\top K_h'Y^\ell\\
				&\quad+\big[B_\ell^\top(P_iX+p_i)+D_\ell^\top P_i(\bar CX+\bar\sigma)\big]^\top K_\ell''Y^h\Big\}\,ds+dM_s^{(2)}.
			\end{aligned}
		\end{equation}
		where \(M^{(2)}\) is a martingale. Since \(P_i(T)=G_i\), \(p_i(T)=g_i\), and \(Y_t^h=Y_t^\ell=Z_t=0\), taking expectations in \eqref{eq:proof-Y-P-Y-full} and \eqref{eq:proof-PX-Z-full} yields
		\begin{equation}\label{eq:proof-terminal-second-expanded}
			\begin{aligned}
				&\mathbb E\big[(Y_T^\ell)^\top G_iY_T^h+(G_iX_T+g_i)^\top Z_T\big]\\
				={}&\mathbb E\int_t^T\Big\{-(Y^\ell)^\top\bar Q_iY^h-(\bar Q_iX+\bar q_i)^\top Z+\big[B_\ell^\top P_iY^h+D_\ell^\top P_i\bar CY^h\big]^\top\phi_\ell''(s,X)\\
				&\quad+\big[B_h^\top P_iY^\ell+D_h^\top P_i\bar CY^\ell\big]^\top\phi_h'(s,X)
				+\phi_\ell''(s,X)^\top D_\ell^\top P_iD_h\phi_h'(s,X)\\
				&\quad+\big[B_h^\top(P_iX+p_i)+D_h^\top P_i(\bar CX+\bar\sigma)\big]^\top K_h'Y^\ell\\
				&\quad+\big[B_\ell^\top(P_iX+p_i)+D_\ell^\top P_i(\bar CX+\bar\sigma)\big]^\top K_\ell''Y^h\Big\}\,ds.
			\end{aligned}
		\end{equation}
		Since
		\begin{equation}\label{eq:proof-Hessian-running-expanded}
			\begin{aligned}
				& (Y^\ell)^\top Q_iY^h+\big[KY^\ell+E_\ell\phi_\ell''(s,X)\big]^\top R_i\big[KY^h+E_h\phi_h'(s,X)\big]\\
				&\quad+(Y^\ell)^\top S_i\big[KY^h+E_h\phi_h'(s,X)\big]+(Y^h)^\top S_i\big[KY^\ell+E_\ell\phi_\ell''(s,X)\big]\\
				={}&(Y^\ell)^\top\bar Q_iY^h+\big[E_\ell^\top(R_iK+S_i^\top)Y^h\big]^\top\phi_\ell''(s,X)+\big[E_h^\top(R_iK+S_i^\top)Y^\ell\big]^\top\phi_h'(s,X)\\
				&\quad+\phi_\ell''(s,X)^\top E_\ell^\top R_iE_h\phi_h'(s,X).
			\end{aligned}
		\end{equation}
		Substituting \eqref{eq:proof-terminal-second-expanded} and \eqref{eq:proof-Hessian-running-expanded} into the second-order probabilistic representation \eqref{eq:second-prob-explicit}, we obtain
		\begin{equation*}
			\begin{aligned}
				&\frac{\delta^2 V_i^{t,x}}{\delta\phi_h\delta\phi_\ell}(\phi;\phi_h',\phi_\ell'')=\mathbb E\bigg[\int_t^T\Big\{\big[Q_iX+S_i(KX+m)+q_i\big]^\top Z\\
				&\quad+\big[R_i(KX+m)+S_i^\top X+\rho_i\big]^\top\big(KZ+E_hK_h'Y^\ell+E_\ell K_\ell''Y^h\big)\\
				&\quad+(Y^\ell)^\top Q_iY^h+\big(KY^\ell+E_\ell\phi_\ell''(s,X)\big)^\top R_i\big(KY^h+E_h\phi_h'(s,X)\big)\\
				&\quad+(Y^\ell)^\top S_i\big(KY^h+E_h\phi_h'(s,X)\big)+(Y^h)^\top S_i\big(KY^\ell+E_\ell\phi_\ell''(s,X)\big)\Big\}\,ds\\
				&\quad+(G_iX_T+g_i)^\top Z_T+(Y_T^\ell)^\top G_iY_T^h\bigg]\\
				={}&\mathbb E\int_t^T\Big\{\big[B_\ell^\top P_iY^h+D_\ell^\top P_i\bar CY^h+E_\ell^\top(R_iK+S_i^\top)Y^h\big]^\top\phi_\ell''(s,X)\\
				&\quad+\big[B_h^\top P_iY^\ell+D_h^\top P_i\bar CY^\ell+E_h^\top(R_iK+S_i^\top)Y^\ell\big]^\top\phi_h'(s,X)\\
				&\quad+\big[B_h^\top(P_iX+p_i)+D_h^\top P_i(\bar CX+\bar\sigma)+E_h^\top(R_i(KX+m)+S_i^\top X+\rho_i)\big]^\top K_h'Y^\ell\\
				&\quad+\big[B_\ell^\top(P_iX+p_i)+D_\ell^\top P_i(\bar CX+\bar\sigma)+E_\ell^\top(R_i(KX+m)+S_i^\top X+\rho_i)\big]^\top K_\ell''Y^h\\
				&\quad+\phi_\ell''(s,X)^\top\big(D_\ell^\top P_iD_h+E_\ell^\top R_iE_h\big)\phi_h'(s,X)\Big\}\,ds.
			\end{aligned}
		\end{equation*}
		By the definitions \eqref{eq:Mih-definition}--\eqref{eq:ellih-definition}, the first four lines of the integrand above are
		\begin{align*}
			&(M_i^\ell Y^h)^\top\phi_\ell''(s,X)+(M_i^hY^\ell)^\top\phi_h'(s,X)+(M_i^hX+\ell_i^h)^\top K_h'Y^\ell+(M_i^\ell X+\ell_i^\ell)^\top K_\ell''Y^h.
		\end{align*}
		Hence
		\begin{equation}\label{eq:proof-second-M-form-expanded}
			\begin{aligned}
				&\frac{\delta^2 V_i^{t,x}}{\delta\phi_h\delta\phi_\ell}(\phi;\phi_h',\phi_\ell'')
				=\mathbb E\int_t^T\Big\{(M_i^\ell Y^h)^\top\phi_\ell''(s,X)+(M_i^hY^\ell)^\top\phi_h'(s,X)\\
				&\qquad+(M_i^hX+\ell_i^h)^\top K_h'Y^\ell+(M_i^\ell X+\ell_i^\ell)^\top K_\ell''Y^h\\
				&\qquad+\phi_h'(s,X)^\top D_h^\top P_iD_\ell\phi_\ell''(s,X)+\phi_\ell''(s,X)^\top E_\ell^\top R_iE_h\phi_h'(s,X)\Big\}\,ds.
			\end{aligned}
		\end{equation}
	    From \eqref{eq:Ph-ode}--\eqref{eq:ph-ode},
		\begin{equation*}
			\begin{aligned}
				&d(P_i^hX+p_i^h)=\Big\{\dot P_i^hX+P_i^h(\bar AX+\bar b)+\dot p_i^h\Big\}\,ds+P_i^h(\bar CX+\bar\sigma)\,dW\\
				&=\Big\{-\bar A^\top(P_i^hX+p_i^h)-\bar C^\top P_i^h(\bar CX+\bar\sigma)-(\Theta_{2,i}^hX+\Theta_{1,i}^h)\Big\}\,ds+P_i^h(\bar CX+\bar\sigma)\,dW.
			\end{aligned}
		\end{equation*}
		Applying It\^o's formula to \((P_i^hX+p_i^h)^\top Y^\ell\), we get
		\begin{equation}\label{eq:proof-cross-h-full}
			\begin{aligned}
				&d\big[(P_i^hX+p_i^h)^\top Y^\ell\big]\\
		        ={}&\Big\{\big[-\bar A^\top(P_i^hX+p_i^h)-\bar C^\top P_i^h(\bar CX+\bar\sigma)-(\Theta_{2,i}^hX+\Theta_{1,i}^h)\big]^\top Y^\ell\\
				&\quad+(P_i^hX+p_i^h)^\top(\bar AY^\ell+B_\ell\phi_\ell''(s,X))+(\bar CX+\bar\sigma)^\top P_i^h(\bar CY^\ell+D_\ell\phi_\ell''(s,X))\Big\}\,ds+dM_s^{(3)}\\
				={}&\Big\{-(\Theta_{2,i}^hX+\Theta_{1,i}^h)^\top Y^\ell+\big[B_\ell^\top(P_i^hX+p_i^h)+D_\ell^\top P_i^h(\bar CX+\bar\sigma)\big]^\top
                \phi_\ell''(s,X)\Big\}\,ds+dM_s^{(3)}.
			\end{aligned}
		\end{equation}
		Since \(P_i^h(T)=p_i^h(T)=0\) and \(Y_t^\ell=0\), taking expectations in \eqref{eq:proof-cross-h-full} yields
		\begin{equation}\label{eq:proof-cross-h-result-expanded}
			\mathbb E\int_t^T\big[B_\ell^\top(P_i^hX+p_i^h)+D_\ell^\top P_i^h(\bar CX+\bar\sigma)\big]^\top\phi_\ell''(s,X)\,ds
			=\mathbb E\int_t^T(\Theta_{2,i}^hX+\Theta_{1,i}^h)^\top Y^\ell\,ds.
		\end{equation}
		Similarly
		\begin{equation*}
			\mathbb E\int_t^T\big[B_h^\top(P_i^\ell X+p_i^\ell)+D_h^\top P_i^\ell(\bar CX+\bar\sigma)\big]^\top\phi_h'(s,X)\,ds
			=\mathbb E\int_t^T(\Theta_{2,i}^\ell X+\Theta_{1,i}^\ell)^\top Y^h\,ds.
		\end{equation*}
		The definitions \eqref{eq:Theta2ih}--\eqref{eq:Theta1ih} imply
		\begin{equation*}
			(\Theta_{2,i}^hX+\Theta_{1,i}^h)^\top Y^\ell
			=(M_i^hY^\ell)^\top\phi_h'(s,X)+(M_i^hX+\ell_i^h)^\top K_h'Y^\ell,
		\end{equation*}
		\begin{equation}\label{eq:proof-source-gradient-ell-expanded}
			(\Theta_{2,i}^\ell X+\Theta_{1,i}^\ell)^\top Y^h
			=(M_i^\ell Y^h)^\top\phi_\ell''(s,X)+(M_i^\ell X+\ell_i^\ell)^\top K_\ell''Y^h.
		\end{equation}
		Combining \eqref{eq:proof-second-M-form-expanded} with \eqref{eq:proof-cross-h-result-expanded}--\eqref{eq:proof-source-gradient-ell-expanded},
		we obtain
		\begin{equation}\label{eq:proof-second-source-expanded-final}
			\begin{aligned}
				&\frac{\delta^2 V_i^{t,x}}{\delta\phi_h\delta\phi_\ell}(\phi;\phi_h',\phi_\ell'')
				=\mathbb E\int_t^T\Big\{\big[B_\ell^\top(P_i^hX+p_i^h)+D_\ell^\top P_i^h(\bar CX+\bar\sigma)\big]^\top\phi_\ell''(s,X)\\
				&\quad+\big[B_h^\top(P_i^\ell X+p_i^\ell)+D_h^\top P_i^\ell(\bar CX+\bar\sigma)\big]^\top\phi_h'(s,X)\\
				&\quad+\phi_h'(s,X)^\top D_h^\top P_iD_\ell\phi_\ell''(s,X)+\phi_\ell''(s,X)^\top E_\ell^\top R_iE_h\phi_h'(s,X)\Big\}\,ds.
			\end{aligned}
		\end{equation}
		The integrand in \eqref{eq:proof-second-source-expanded-final} is exactly \(\mathscr R_i^{h,\ell}(s,X_s)\) by \eqref{eq:source-second-expanded-quadratic}. Therefore
		\begin{equation}\label{eq:proof-second-prob-source-expanded}
			\frac{\delta^2 V_i^{t,x}}{\delta\phi_h\delta\phi_\ell}(\phi;\phi_h',\phi_\ell'')=\mathbb E\int_t^T\mathscr R_i^{h,\ell}(s,X_s)\,ds.
		\end{equation}
		Finally, applying It\^o's formula to \(z_i^{h,\ell}(s,X_s)\) in \eqref{eq:z-quadratic-representation}, and using \eqref{eq:Phl-ode}--\eqref{eq:rhl-ode} together with \(z_i^{h,\ell}(T,\cdot)=0\), gives
		\[
		z_i^{h,\ell}(t,x)=\mathbb E\int_t^T\mathscr R_i^{h,\ell}(s,X_s)\,ds.
		\]
		Together with \eqref{eq:proof-second-prob-source-expanded}, this proves \eqref{eq:second-equivalence-clean}. The proof is complete.
	\end{proof}
	
	\section{Open-loop stochastic LQ $\alpha$-potential game}
	
	This section develops the open-loop LQ $\alpha$-potential framework. We construct the $\alpha$-potential function, estimate the parameter $\alpha$, and use the resulting finite-dimensional minimization problem to obtain an approximate open-loop Nash equilibrium.
	
	Throughout this section, let $(\Omega,\mathcal F,\mathbb F,\mathbb P)$ be a complete filtered probability space satisfying the usual conditions and supporting a \(d_W\)-dimensional Brownian motion $W\equiv(W^1,\ldots,W^{d_W})^\top$. For \(x\in\mathbb R^n\) and \(u\in\mathbb R^k\), write $x\equiv((x^1)^\top,\ldots,(x^N)^\top)^\top, x^i\in\mathbb R^{n_i}$, and $u\equiv(u_1^\top,\ldots,u_N^\top)^\top, u_i\in\mathbb R^{k_i}$. Here \(n_i\) and \(k_i\) denote the state and control dimensions of player \(i\), respectively. For each $i\in\mathcal I_N$, let $A_i\subset\mathbb R^{k_i}$ be the open-loop action set of player $i$, $\mathbb A:=\prod_{i=1}^N A_i\subset\mathbb R^k .$ Define
$\mathcal U_i:=\left\{u_i:\Omega\times[t,T]\to A_i\ \middle|\ u_i\ \text{is }\mathbb F\text{-progressively measurable and }u_i\in \mathcal H^2([t,T];\mathbb R^{k_i})\right\}$.
Set $\mathcal U^{(N)}:=\prod_{i=1}^N\mathcal U_i$.
	
	For $u\in\mathcal U^{(N)}$, the controlled state equation is
	\begin{equation}\label{eq:open-loop-state}
		\left\{
		\begin{aligned}
			dX_s^u={}&(A_sX_s^u+B_su_s+b_s)\,ds+\sum_{h=1}^{d_W}\bigl(C_s^hX_s^u+D_s^hu_s+\sigma_s^h\bigr)\,dW_s^h,\qquad s\in[t,T],\\
			X_t^u={}&x.
		\end{aligned}
		\right.
	\end{equation}
	The cost function of player $i$ is given by 
	\begin{equation*}
		\begin{aligned}
			J_i^{t,x}(u):=\frac12\mathbb E\bigg[&\int_t^T\Big((X_s^u)^\top Q_i(s)X_s^u+u_s^\top R_i(s)u_s+2(X_s^u)^\top S_i(s)u_s \\
			&\qquad+2q_i(s)^\top X_s^u+2\rho_i(s)^\top u_s\Big)ds+(X_T^u)^\top G_iX_T^u+2g_i^\top X_T^u\bigg].
		\end{aligned}
	\end{equation*}
	Assume that \(A,B,b\) and the cost coefficients satisfy the corresponding conditions in Assumption~\ref{ass:LQ-data}. For each \(h=1,\ldots,d_W\), assume that
	\[
	C^h\in L^\infty(0,T;\mathbb R^{n\times n}),\qquad D^h\in L^\infty(0,T;\mathbb R^{n\times k}),\qquad \sigma^h\in L^2(0,T;\mathbb R^n).
	\]
	We also assume that each action set $A_i$ is convex and $0\in A_i$ for all $i\in\mathcal I_N$.
	
	Let
	\[
	B_i(s):=B_sE_i,\qquad D_i^h(s):=D_s^hE_i,\qquad h=1,\ldots,d_W,
	\]
	so that
	\[
	B_su_s=\sum_{i=1}^N B_i(s)u_{i,s},\qquad D_s^hu_s=\sum_{i=1}^N D_i^h(s)u_{i,s}.
	\]

	For a direction \(u_i'\in \mathcal H^2([t,T];\mathbb R^{k_i})\), let $Y^{i,u_i'}$ be the solution of the first-order variational equation
	\begin{equation*}
		\left\{
		\begin{aligned}
			dY_s^{i,u_i'}={}&(A_sY_s^{i,u_i'}+B_i(s)u_{i,s}')\,ds+\sum_{h=1}^{d_W}\bigl(C_s^hY_s^{i,u_i'}+D_i^h(s)u_{i,s}'\bigr)\,dW_s^h,\\
			Y_t^{i,u_i'}={}&0.
		\end{aligned}
		\right.
	\end{equation*}
	We write $\bar Y^u:=\sum_{i=1}^NY^{i,u_i}$.
	\begin{mylem}\label{lem:Xru-decomposition}
		For every $u\in\mathcal U^{(N)}$ and every $r\in[0,1]$,
		\[
		X_s^{ru}=X_s^u-(1-r)\bar Y_s^u,\qquad s\in[t,T].
		\]
	\end{mylem}
	
	\begin{proof}
		Set $Z_s^r:=X_s^u-(1-r)\bar Y_s^u$. By the state equation and the variational equations,
		\[
		dZ_s^r=\bigl(A_sZ_s^r+rB_su_s+b_s\bigr)\,ds+\sum_{h=1}^{d_W}\bigl(C_s^hZ_s^r+rD_s^hu_s+\sigma_s^h\bigr)\,dW_s^h,\qquad Z_t^r=x.
		\]
		Thus $Z^r$ satisfies \eqref{eq:open-loop-state} with control $ru$. The conclusion follows from uniqueness of the state equation.
	\end{proof}

    Similarly to Proposition~\ref{First-order probabilistic representation} and Proposition~\ref{Second-order probabilistic representation}, we have
	\begin{equation}\label{eq:open-loop-first-derivative}
		\begin{aligned}
			\frac{\delta J_i^{t,x}}{\delta u_i}(u;u_i')={}&\mathbb E\bigg[\int_t^T\Big\{\big(Q_iX_s^u+S_i u_s+q_i\big)^\top Y_s^{i,u_i'}\\
			&\qquad\qquad+\big(R_i u_s+S_i^\top X_s^u+\rho_i\big)^\top E_i u_{i,s}'\Big\}\,ds+(G_iX_T^u+g_i)^\top Y_T^{i,u_i'}\bigg].
		\end{aligned}
	\end{equation}
    \begin{equation*}
	    \begin{aligned}
	        \frac{\delta^2 J_i^{t,x}}{\delta u_i\delta u_j}(u;u_i',u_j'')={}
            &\mathbb E\bigg[\int_t^T\Big\{\big(Y_s^{j,u_j''}\big)^\top Q_iY_s^{i,u_i'}+\big(Y_s^{j,u_j''}\big)^\top S_iE_i u_{i,s}'\\
		    &\qquad\qquad+\big(Y_s^{i,u_i'}\big)^\top S_iE_j u_{j,s}''+\big(E_j u_{j,s}''\big)^\top R_iE_i u_{i,s}'\Big\}\,ds
            +\big(Y_T^{j,u_j''}\big)^\top G_iY_T^{i,u_i'}\bigg].
	    \end{aligned}
    \end{equation*}
	Define the open-loop $\alpha$-potential function
	\begin{equation*}
		\Phi^{t,x}(u):=\int_0^1\sum_{i=1}^N\frac{\delta J_i^{t,x}}{\delta u_i}(ru;u_i)\,dr .
	\end{equation*}
	Using \eqref{eq:open-loop-first-derivative} and Lemma~\ref{lem:Xru-decomposition}, we have:
	\begin{equation*}
		\begin{aligned}
			\Phi^{t,x}(u)={}&\int_0^1\sum_{i=1}^N\mathbb E\bigg[\int_t^T\Big\{\big[Q_i(X_s^u-(1-r)\bar Y_s^u)+S_i(ru_s)+q_i\big]^\top Y_s^{i,u_i}\\
			&\qquad\qquad+\big[R_i(ru_s)+S_i^\top(X_s^u-(1-r)\bar Y_s^u)+\rho_i\big]^\top E_iu_{i,s}\Big\}\,ds\\
			&\qquad+\big[G_i(X_T^u-(1-r)\bar Y_T^u)+g_i\big]^\top Y_T^{i,u_i}\bigg]dr.
		\end{aligned}
	\end{equation*}
	Applying Fubini's theorem, we get
	\begin{equation}\label{eq:open-loop-potential-explicit}
		\begin{aligned}
			\Phi^{t,x}(u)={}&\mathbb E\Bigg[\int_t^T\sum_{i=1}^N\Big\{\big[Q_i(X_s^u-\tfrac12\bar Y_s^u)+\tfrac12S_iu_s+q_i\big]^\top Y_s^{i,u_i}\\
			&\qquad\qquad+\big[\tfrac12R_iu_s+S_i^\top(X_s^u-\tfrac12\bar Y_s^u)+\rho_i\big]^\top E_iu_{i,s}\Big\}\,ds\\
			&\qquad+\sum_{i=1}^N\big[G_i(X_T^u-\tfrac12\bar Y_T^u)+g_i\big]^\top Y_T^{i,u_i}\Bigg].
		\end{aligned}
	\end{equation}
	
	For
	\[
	y\equiv(y_1,\ldots,y_N)\in(\mathbb R^n)^N,\qquad\bar y:=\sum_{i=1}^Ny_i,
	\]
define
	\begin{equation}\label{eq:F-open-loop-explicit}
		\begin{aligned}
			F(t,x,y,a):=\sum_{i=1}^N\Bigg\{&y_i^\top Q_i(t)\left(x-\frac12\bar y\right)+\frac12 y_i^\top S_i(t)a+q_i(t)^\top y_i\\
			&+(E_i a_i)^\top\left[\frac12 R_i(t)a+S_i(t)^\top\left(x-\frac12\bar y\right)+\rho_i(t)\right]\Bigg\},
		\end{aligned}
	\end{equation}
and
	\begin{equation}\label{eq:G-open-loop-explicit}
		G(x,y):=\sum_{i=1}^N\left[y_i^\top G_i\left(x-\frac12\bar y\right)+g_i^\top y_i\right].
	\end{equation}
	
	\begin{mypro}
		For every $u\in\mathcal U^{(N)}$,
		\begin{equation*}
			\Phi^{t,x}(u)=\mathbb E\left[\int_t^T F(s,X_s^u,Y_s^{1,u_1},\ldots,Y_s^{N,u_N},u_s)\,ds+G(X_T^u,Y_T^{1,u_1},\ldots,Y_T^{N,u_N})\right].
		\end{equation*}
	\end{mypro}
	
	\begin{proof}
		The result follows directly from \eqref{eq:open-loop-potential-explicit} and the definitions \eqref{eq:F-open-loop-explicit}--\eqref{eq:G-open-loop-explicit}.
	\end{proof}

\subsection{Estimate of $\alpha$ for the open-loop LQ game}

\begin{mylem}
	\label{lem:blockwise-first-variation}
	Write $Y^{i,v_i}\equiv((Y^{i,v_i,1})^\top,\ldots,(Y^{i,v_i,N})^\top)^\top.$ Decompose the matrix according to the state blocks,
	\[
	A_s=\bigl(A^{r\ell}(s)\bigr)_{r,\ell=1}^N,\qquad A^{r\ell}(s)\in\mathbb R^{n_r\times n_\ell},
	\]
	and, for each \(h=1,\ldots,d_W\), write
	\[
	C_s^h=\bigl(C^{h,r\ell}(s)\bigr)_{r,\ell=1}^N,\qquad C^{h,r\ell}(s)\in\mathbb R^{n_r\times n_\ell}.
	\]
	Similarly, write
	\[
	B_i(s)\equiv
		\begin{pmatrix}
		B_i^1(s)\\
		\vdots\\
		B_i^N(s)
		\end{pmatrix},
	\qquad D_i^h(s)\equiv
		\begin{pmatrix}
		D_i^{h,1}(s)\\
		\vdots\\
		D_i^{h,N}(s)
		\end{pmatrix},
	\]
where
	\[
	B_i^r(s),D_i^{h,r}(s)\in\mathbb R^{n_r\times k_i}.
	\]
	Let \(Y^{i,v_i}\) be the first variation generated by \(v_i\in\mathcal H^2([t,T];\mathbb R^{k_i})\). Then its \(r\)-th block satisfies
	\[
	\begin{aligned}
	    dY_s^{i,v_i,r}={}&\left(\sum_{\ell=1}^N A^{r\ell}(s)Y_s^{i,v_i,\ell}+B_i^r(s)v_{i,s}\right)ds\\
	    &+\sum_{h=1}^{d_W}\left(\sum_{\ell=1}^N C^{h,r\ell}(s)Y_s^{i,v_i,\ell}+D_i^{h,r}(s)v_{i,s}\right)dW_s^h,\qquad Y_t^{i,v_i,r}=0.
	\end{aligned}
	\]
	For \(r,\ell=1,\ldots,N\), set
	\[
	\begin{aligned}
	    \gamma_{r\ell}:={}&\mathbf 1_{\{r=\ell\}}\left(1+2\|A^{rr}\|_{L^\infty(0,T)}+\sum_{m\ne r}\|A^{rm}\|_{L^\infty(0,T)}\right)	\\
	    &+\mathbf 1_{\{r\ne \ell\}}\|A^{r\ell}\|_{L^\infty(0,T)}+2\sum_{h=1}^{d_W}\left(\sum_{m=1}^N\|C^{h,rm}\|_{L^\infty(0,T)}\right)\|C^{h,r\ell}\|_{L^\infty(0,T)}.
	\end{aligned}
	\]
	Define
	\[
	\gamma:=\max_{1\le \ell\le N}\sum_{r=1}^N\gamma_{r\ell}.
	\]
	For \(r=1,\ldots,N\) and \(i\in\mathcal I_N\), set
	\[
	\begin{aligned}
	    \chi_{ri}:=e^{\gamma_{rr}T}\Bigg[&\|B_i^r\|_{L^\infty(0,T)}^2+2\sum_{h=1}^{d_W}\|D_i^{h,r}\|_{L^\infty(0,T)}^2\\
	    &+T\Bigg(\sum_{\ell\ne r}\gamma_{r\ell}\Bigg)e^{\gamma T}\sum_{m=1}^N\left(\|B_i^m\|_{L^\infty(0,T)}^2+2\sum_{h=1}^{d_W}\|D_i^{h,m}\|_{L^\infty(0,T)}^2\right)\Bigg].
	\end{aligned}
	\]
	Then, for every \(s\in[t,T]\),
	\[
	\mathbb E|Y_s^{i,v_i,r}|^2\le\chi_{ri}\|v_i\|_{\mathcal H^2([t,T];\mathbb R^{k_i})}^2 .
	\]
	Consequently,
	\[
	\|Y^{i,v_i,r}\|_{\mathcal H^2([t,T];\mathbb R^{n_r})}^2\le T\chi_{ri}\|v_i\|_{\mathcal H^2([t,T];\mathbb R^{k_i})}^2,
	\]
	and
	\[
	\mathbb E|Y_T^{i,v_i,r}|^2\le\chi_{ri}\|v_i\|_{\mathcal H^2([t,T];\mathbb R^{k_i})}^2 .
	\]
\end{mylem}

\begin{proof}
	By Itô's formula,
	\[
	\begin{aligned}
	    \frac{d}{ds}\mathbb E|Y_s^{i,v_i,r}|^2={}&2\mathbb E\left\langle Y_s^{i,v_i,r},\sum_{\ell=1}^N A^{r\ell}(s)Y_s^{i,v_i,\ell}+B_i^r(s)v_{i,s}\right\rangle\\
	    &+\sum_{h=1}^{d_W}\mathbb E\left|\sum_{\ell=1}^N C^{h,r\ell}(s)Y_s^{i,v_i,\ell}+D_i^{h,r}(s)v_{i,s}\right|^2 .
	\end{aligned}
	\]
	The first term is bounded by
	\[
	\begin{aligned}
	    &2\left\langle Y_s^{i,v_i,r},\sum_{\ell=1}^N A^{r\ell}(s)Y_s^{i,v_i,\ell}+B_i^r(s)v_{i,s}\right\rangle\\
	    &\quad\le\Bigg(1+2\|A^{rr}\|_{L^\infty(0,T)}+\sum_{\ell\ne r}\|A^{r\ell}\|_{L^\infty(0,T)}\Bigg)|Y_s^{i,v_i,r}|^2\\
	    &\qquad+\sum_{\ell\ne r}\|A^{r\ell}\|_{L^\infty(0,T)}|Y_s^{i,v_i,\ell}|^2+\|B_i^r\|_{L^\infty(0,T)}^2|v_{i,s}|^2 .
	\end{aligned}
	\]
	For every \(h=1,\ldots,d_W\),
	\[
	\left|\sum_{\ell=1}^N C^{h,r\ell}(s)y^\ell\right|^2\le\left(\sum_{\ell=1}^N\|C^{h,r\ell}\|_{L^\infty(0,T)}\right)\sum_{\ell=1}^N\|C^{h,r\ell}\|_{L^\infty(0,T)}|y^\ell|^2.
    \]
	Hence
	\[
	\begin{aligned}
	    &\sum_{h=1}^{d_W}\left|\sum_{\ell=1}^N C^{h,r\ell}(s)Y_s^{i,v_i,\ell}+D_i^{h,r}(s)v_{i,s}\right|^2\\
	    &\quad\le 2\sum_{h=1}^{d_W}\left(\sum_{m=1}^N\|C^{h,rm}\|_{L^\infty(0,T)}\right)\sum_{\ell=1}^N\|C^{h,r\ell}\|_{L^\infty(0,T)}|Y_s^{i,v_i,\ell}|^2
	    +2\sum_{h=1}^{d_W}\|D_i^{h,r}\|_{L^\infty(0,T)}^2|v_{i,s}|^2.
	\end{aligned}
	\]
	Therefore,
	\[
	\begin{aligned}
	    \frac{d}{ds}\mathbb E|Y_s^{i,v_i,r}|^2\le{}\sum_{\ell=1}^N\gamma_{r\ell}\mathbb E|Y_s^{i,v_i,\ell}|^2
        +\left(\|B_i^r\|_{L^\infty(0,T)}^2+2\sum_{h=1}^{d_W}\|D_i^{h,r}\|_{L^\infty(0,T)}^2\right)\mathbb E|v_{i,s}|^2.
	\end{aligned}
	\]
	Summing over \(r=1,\ldots,N\), we obtain
	\[
	\begin{aligned}
	    \frac{d}{ds}\sum_{r=1}^N\mathbb E|Y_s^{i,v_i,r}|^2\le{}\gamma\sum_{r=1}^N\mathbb E|Y_s^{i,v_i,r}|^2
        +\left[\sum_{r=1}^N\left(\|B_i^r\|_{L^\infty(0,T)}^2+2\sum_{h=1}^{d_W}\|D_i^{h,r}\|_{L^\infty(0,T)}^2\right)\right]\mathbb E|v_{i,s}|^2 .
	\end{aligned}
	\]
	Since \(Y_t^{i,v_i}=0\), Gronwall's inequality yields
	\begin{equation}\label{eq:total sum}
	\begin{aligned}
	    \sum_{r=1}^N\mathbb E|Y_s^{i,v_i,r}|^2\le{}&e^{\gamma T}\left[\sum_{r=1}^N\left(\|B_i^r\|_{L^\infty(0,T)}^2
        +2\sum_{h=1}^{d_W}\|D_i^{h,r}\|_{L^\infty(0,T)}^2\right)\right]\|v_i\|_{\mathcal H^2([t,T];\mathbb R^{k_i})}^2 .
	\end{aligned}
	\end{equation}
	For a fixed \(r\), the preceding inequality also implies
	\[
	\begin{aligned}
	    \frac{d}{ds}\mathbb E|Y_s^{i,v_i,r}|^2\le{}&\gamma_{rr}\mathbb E|Y_s^{i,v_i,r}|^2+\Bigg(\sum_{\ell\ne r}\gamma_{r\ell}\Bigg)\sum_{m=1}^N\mathbb E|Y_s^{i,v_i,m}|^2\\
	    &+\left(\|B_i^r\|_{L^\infty(0,T)}^2+2\sum_{h=1}^{d_W}\|D_i^{h,r}\|_{L^\infty(0,T)}^2\right)\mathbb E|v_{i,s}|^2.
	\end{aligned}
	\]
	Applying Gronwall's inequality again and using \eqref{eq:total sum}, we obtain
	\[
	\begin{aligned}
	    \mathbb E|Y_s^{i,v_i,r}|^2\le{}&e^{\gamma_{rr}T}\Bigg[\|B_i^r\|_{L^\infty(0,T)}^2+2\sum_{h=1}^{d_W}\|D_i^{h,r}\|_{L^\infty(0,T)}^2\\
	    &\quad+T\Bigg(\sum_{\ell\ne r}\gamma_{r\ell}\Bigg)e^{\gamma T}\sum_{m=1}^N\left(\|B_i^m\|_{L^\infty(0,T)}^2
        +2\sum_{h=1}^{d_W}\|D_i^{h,m}\|_{L^\infty(0,T)}^2\right)\Bigg]\|v_i\|_{\mathcal H^2([t,T];\mathbb R^{k_i})}^2\\
	    ={}&\chi_{ri}\|v_i\|_{\mathcal H^2([t,T];\mathbb R^{k_i})}^2 .
	\end{aligned}
	\]
	Integrating over \(s\in[t,T]\) gives the \(\mathcal H^2\)-estimate, and taking \(s=T\) gives the terminal estimate.
\end{proof}

\begin{mypro}[Estimate of \(\alpha\)]\label{pro:open-loop-LQ-alpha}
	In the open-loop LQ setting above, assume that the admissible control class is \(\mathcal H^2\)-bounded, namely, there exists \(L>0\) such that
	\[
	\sup_{i\in\mathcal I_N}\sup_{u_i\in\mathcal U_i}\|u_i\|_{\mathcal H^2([t,T];\mathbb R^{k_i})}\le L .
	\]
	Decompose the state space as above. Accordingly, write
	\[
	Q_i=\bigl(Q_i^{r\ell}\bigr)_{r,\ell=1}^N,\qquad G_i=\bigl(G_i^{r\ell}\bigr)_{r,\ell=1}^N,
	\]
	where
	\[
	Q_i^{r\ell}\in\mathbb R^{n_r\times n_\ell},\qquad G_i^{r\ell}\in\mathbb R^{n_r\times n_\ell}.
	\]
    Also write
    \[
    S_i=\big(S_i^{ra}\big)_{r,a=1}^N,\qquad S_i^{ra}\in\mathbb R^{n_r\times k_a},\qquad S_i^a:=S_iE_a=
        \begin{pmatrix}
    	S_i^{1a}\\
    	\vdots\\
	    S_i^{Na}
       \end{pmatrix}
    \in\mathbb R^{n\times k_a}.
    \]
    Moreover
	\[
	R_i=\bigl(R_i^{ab}\bigr)_{a,b=1}^N,\qquad R_i^{ab}\in\mathbb R^{k_a\times k_b},\qquad \rho_i=
	    \begin{pmatrix}
		\rho_i^1\\
		\vdots\\
		\rho_i^N
	    \end{pmatrix},\qquad\rho_i^j\in\mathbb R^{k_j},
	\]
	For \(i\ne j\), define
	\[
	\begin{aligned}
	    \Lambda_{ij}:={}&T\sum_{r,\ell=1}^N\sqrt{\chi_{rj}\chi_{\ell i}}\,\|Q_i^{r\ell}-Q_j^{r\ell}\|_{L^\infty(0,T)}\\
	    &+\sqrt T\sum_{r=1}^N\sqrt{\chi_{rj}}\,\|S_i^{r,i}-S_j^{r,i}\|_{L^\infty(0,T)}+\sqrt T\sum_{r=1}^N\sqrt{\chi_{ri}}\,\|S_i^{r,j}-S_j^{r,j}\|_{L^\infty(0,T)}\\
	    &+\|R_i^{ji}-R_j^{ji}\|_{L^\infty(0,T)}+\sum_{r,\ell=1}^N\sqrt{\chi_{rj}\chi_{\ell i}}\,\|G_i^{r\ell}-G_j^{r\ell}\| .
	\end{aligned}
	\]
    Then the game is an open-loop $\alpha$-potential game with
	\begin{equation}\label{eq:open-loop-alpha-estimate}
		\alpha\le L^2\max_{1\le i\le N}\sum_{j\ne i}\Lambda_{ij}.
	\end{equation}
\end{mypro}

\begin{proof}
	For \(i\ne j\), take \(v_i\in\mathcal U_i\) and \(v_j\in\mathcal U_j\). Using the block decompositions, the mixed second-order derivative difference can be written as
	\[
	\begin{aligned}
	    &\frac{\delta^2J_i^{t,x}}{\delta u_i\delta u_j}(u;v_i,v_j)-\frac{\delta^2J_j^{t,x}}{\delta u_j\delta u_i}(u;v_j,v_i)\\
	    ={}&\mathbb E\int_t^T\Bigg[\sum_{r,\ell=1}^N\left\langle Y_s^{j,v_j,r},\bigl(Q_i^{r\ell}(s)-Q_j^{r\ell}(s)\bigr)Y_s^{i,v_i,\ell}\right\rangle
	    +\sum_{r=1}^N\left\langle Y_s^{j,v_j,r},\bigl(S_i^{r,i}(s)-S_j^{r,i}(s)\bigr)v_{i,s}\right\rangle\\
	    &\qquad+\sum_{r=1}^N\left\langle Y_s^{i,v_i,r},\bigl(S_i^{r,j}(s)-S_j^{r,j}(s)\bigr)v_{j,s}\right\rangle
	    +v_{j,s}^\top\bigl(R_i^{ji}(s)-R_j^{ji}(s)\bigr)v_{i,s}\Bigg]ds\\
	    &\quad+\mathbb E\left[\sum_{r,\ell=1}^N\left\langle Y_T^{j,v_j,r},\bigl(G_i^{r\ell}-G_j^{r\ell}\bigr)Y_T^{i,v_i,\ell}\right\rangle\right].
	\end{aligned}
	\]
	Therefore
	\[
	\begin{aligned}
	    &\left|\frac{\delta^2J_i^{t,x}}{\delta u_i\delta u_j}(u;v_i,v_j)-\frac{\delta^2J_j^{t,x}}{\delta u_j\delta u_i}(u;v_j,v_i)\right|\\
	    \le{}&\sum_{r,\ell=1}^N\|Q_i^{r\ell}-Q_j^{r\ell}\|_{L^\infty(0,T)}\|Y^{j,v_j,r}\|_{\mathcal H^2}\|Y^{i,v_i,\ell}\|_{\mathcal H^2}
        +\sum_{r=1}^N\|S_i^{r,i}-S_j^{r,i}\|_{L^\infty(0,T)}\|Y^{j,v_j,r}\|_{\mathcal H^2}\|v_i\|_{\mathcal H^2}\\
        &+\sum_{r=1}^N\|S_i^{r,j}-S_j^{r,j}\|_{L^\infty(0,T)}\|Y^{i,v_i,r}\|_{\mathcal H^2}\|v_j\|_{\mathcal H^2}
        +\|R_i^{ji}-R_j^{ji}\|_{L^\infty(0,T)}\|v_i\|_{\mathcal H^2}\|v_j\|_{\mathcal H^2}\\
	    &+\sum_{r,\ell=1}^N\|G_i^{r\ell}-G_j^{r\ell}\|\|Y_T^{j,v_j,r}\|_{L^2(\Omega)}\|Y_T^{i,v_i,\ell}\|_{L^2(\Omega)}.
	\end{aligned}
	\]
	By Lemma~\ref{lem:blockwise-first-variation},
	\[
	\|Y^{i,v_i,r}\|_{\mathcal H^2}\le\sqrt{T\chi_{ri}}\,\|v_i\|_{\mathcal H^2},\qquad\|Y_T^{i,v_i,r}\|_{L^2(\Omega)}\le\sqrt{\chi_{ri}}\,\|v_i\|_{\mathcal H^2}.
	\]
	Thus
	\[
	\left|\frac{\delta^2J_i^{t,x}}{\delta u_i\delta u_j}(u;v_i,v_j)-\frac{\delta^2J_j^{t,x}}{\delta u_j\delta u_i}(u;v_j,v_i)\right|
	\le\Lambda_{ij}\|v_i\|_{\mathcal H^2}\|v_j\|_{\mathcal H^2}.
	\]
	Since the admissible controls are \(\mathcal H^2\)-bounded by \(L\), \eqref{eq:alpha-estimate} gives
	\[
	\alpha\le L^2\max_{1\le i\le N}\sum_{j\ne i}\Lambda_{ij}.
	\]
    The proof is complete.
\end{proof}

\begin{Remark}
	The blockwise form in Proposition~\ref{pro:open-loop-LQ-alpha} is used  for the estimate of \(\alpha\). The reason for keeping the state blocks is that the mixed second-order linear derivative involves the products of the first variations generated by two different players. For example, the term
	\[
	(Y_s^{j,v_j})^\top(Q_i(s)-Q_j(s))Y_s^{i,v_i}
	\]
is estimated as
	\[
	\sum_{r,\ell=1}^N\left\langle Y_s^{j,v_j,r},\bigl(Q_i^{r\ell}(s)-Q_j^{r\ell}(s)\bigr)Y_s^{i,v_i,\ell}\right\rangle.
	\]
	Hence only those blocks that interact with the nonzero components of the variations \(Y^{j,v_j}\) and \(Y^{i,v_i}\) contribute to the estimate. If one estimates this term directly by the full matrix norm
	\[
	\|Q_i-Q_j\|\,\|Y^{j,v_j}\|\,\|Y^{i,v_i}\|,
	\]
then the block structure is lost, and blocks which do not actually contribute to the mixed derivative are also included in the bound. 
\end{Remark}

   The next example shows that the estimate of $\alpha$ in Proposition~\ref{pro:open-loop-LQ-alpha} is not merely a special case of that in \cite{GuoLiZhangBSDE2025}.  
In fact, it can be strictly sharper in some cases.
\begin{example}
	\label{ex:sparse-directed-comparison}	
	Let \(N\geq5\), and let every state and control component be one-dimensional. Assume that
	\[
	\sup_{i\in\mathcal I_N}\sup_{u_i\in\mathcal U_i}\|u_i\|_{\mathcal H^2([t,T])}\leq L,
	\]
	Consider
	\[
	\left\{
	\begin{aligned}
		dX_s^1&=\left(\frac{\kappa}{N^2}X_s^2+u_{1,s}\right)ds+\sigma_1\,dW_s^1,\\
		dX_s^r&=u_{r,s}\,ds+\sigma_r\,dW_s^r,\qquad r=2,\ldots,N,
	\end{aligned}
	\right.
	\]
	where \(\kappa>0\) is independent of \(N\). The cost functions are
	\[
	J_1^{t,x}(u)=\frac12\mathbb E\int_t^T u_{1,s}^2\,ds,
	\]
	\[
	J_2^{t,x}(u)=\frac12\mathbb E\int_t^T\left[(X_s^2)^2+u_{2,s}^2+\frac1{N^2}(X_s^2-X_s^3)^2\right]ds,
	\]
	\[
	J_i^{t,x}(u)=\frac12\mathbb E\int_t^T\left[(X_s^i)^2+u_{i,s}^2\right]ds,\qquad i=3,\ldots,N.
	\]
Hence
	\[
	A^{12}=\frac{\kappa}{N^2},\qquad A^{r\ell}=0\quad\text{for }(r,\ell)\neq(1,2),
	\]
	and
	\[
	B_i^r=\delta_{ri},\qquad C^{h,r\ell}=D_i^{h,r}=0.
	\]
	The nonzero state-cost coefficients are
	\[
	Q_2^{22}=1+\frac1{N^2},\qquad Q_2^{23}=Q_2^{32}=-\frac1{N^2},\qquad Q_2^{33}=\frac1{N^2},
	\]
	and
	\[
	Q_i^{ii}=1,\qquad i=3,\ldots,N.
	\]
Moreover,
	\[
	R_i^{ii}=1,\qquad S_i=G_i=0,
	\]
	with all other blocks of \(R_i\) equal to zero.
	
	The constants in Lemma~\ref{lem:blockwise-first-variation} satisfy
	\[
	\gamma_{11}=1+\frac{\kappa}{N^2},\qquad\gamma_{12}=\frac{\kappa}{N^2},\qquad\gamma_{rr}=1,\quad r\geq2,
	\]
	Hence
	\[
	\chi_{ri}=e^T\delta_{ri},\qquad r=2,\ldots,N.
	\]

	Since \(S_i=G_i=0\) and
	\[
	R_i^{ji}-R_j^{ji}=0,\qquad i\neq j,
	\]
	the constant in Proposition~\ref{pro:open-loop-LQ-alpha} reduces, for \(i,j\geq2\), to
	\[
	\Lambda_{ij}=Te^T|Q_i^{ji}-Q_j^{ji}|.
	\]
	
	Therefore,
	\[
	\Lambda_{23}=\Lambda_{32}=\frac{Te^T}{N^2},\qquad\Lambda_{ij}=0\quad\text{for all other }i\neq j.
	\]
	Consequently,
	\begin{equation*}
		\alpha_N\leq L^2\max_{i\in\mathcal I_N}\sum_{j\neq i}\Lambda_{ij}=\frac{L^2Te^T}{N^2}.
	\end{equation*}
	
	We compare this result with \cite{GuoLiZhangBSDE2025}. The only nonzero derivative of the drift is $\partial_{x_2}b_1=\frac{\kappa}{N^2}.$ So we take $L_y^b=\frac{\kappa}{N}$. Substituting the coefficients yields for all \(i\neq j\),
	\[
	\widetilde C^{i,j}=O(N^{-2}),\qquad\max_{i\in\mathcal I_N}\sum_{j\neq i}\widetilde C^{i,j}=O(N^{-1}).
	\]
	Thus the estimate obtained from the approach in \cite{GuoLiZhangBSDE2025} is of order \(O(N^{-1})\).
\end{example}

\subsection{Reduction to a finite-dimensional control problem}
	
	Define the extended state process
	\[
	\mathbf X_s^u:=\big(X_s^u,Y_s^{1,u_1},\ldots,Y_s^{N,u_N}\big)\in\mathbb R^{n(N+1)}.
	\]
	For
	\[
	\mathbf x\equiv(x,y_1,\ldots,y_N)\in\mathbb R^{n(N+1)},\qquad a\equiv(a_1^\top,\ldots,a_N^\top)^\top\in\mathbb A,
	\]
	define
	\[
	\mu(t,\mathbf x,a):=
	\begin{pmatrix}
		A_tx+B_ta+b_t\\
		A_ty_1+B_1(t)a_1\\
		\vdots\\
		A_ty_N+B_N(t)a_N
	\end{pmatrix},
	\]
	and, for each \(h=1,\ldots,d_W\),
	\[
	\Lambda^h(t,\mathbf x,a):=
	\begin{pmatrix}
		C_t^hx+D_t^ha+\sigma_t^h\\
		C_t^hy_1+D_1^h(t)a_1\\
		\vdots\\
		C_t^hy_N+D_N^h(t)a_N
	\end{pmatrix},
	\]
	$$
	\Lambda(t,\mathbf x,a)=\bigl(\Lambda^1(t,\mathbf x,a),\ldots,\Lambda^{d_W}(t,\mathbf x,a)\bigr)\in\mathbb R^{n(N+1)\times d_W}.
	$$
	For any $(t,\mathbf x)\in[0,T]\times\mathbb R^{n(N+1)}$ and $u\in\mathcal U^{(N)}$, the corresponding extended state $\mathbf X^{t,\mathbf x,u}$ satisfies
	\begin{equation}\label{eq:extended-state-SDE}
		\left\{
		\begin{aligned}
			d\mathbf X_s^{t,\mathbf x,u}={}&\mu(s,\mathbf X_s^{t,\mathbf x,u},u_s)\,ds+\sum_{h=1}^{d_W}\Lambda^h(s,\mathbf X_s^{t,\mathbf x,u},u_s)\,dW_s^h,\qquad s\in[t,T],\\
			\mathbf X_t^{t,\mathbf x,u}={}&\mathbf x .
		\end{aligned}
		\right.
	\end{equation}
    For each admissible control $u\in\mathcal U^{(N)}$, \eqref{eq:extended-state-SDE} admits a unique strong solution in $\mathcal S^2([t,T];\mathbb R^{n(N+1)})$.
	
	Define
	\begin{equation*}
		\mathcal L^a\varphi(t,\mathbf x):=\mu(t,\mathbf x,a)^\top \partial_{\mathbf x}\varphi(t,\mathbf x)
        +\frac12\operatorname{tr}\bigl(\Lambda(t,\mathbf x,a)\Lambda(t,\mathbf x,a)^\top \partial_{\mathbf x\mathbf x}^2\varphi(t,\mathbf x)\bigr).
	\end{equation*}
	Here and below, $\partial_{\mathbf x}$ and $\partial_{\mathbf x\mathbf x}^2$ denote the gradient and Hessian with respect to the extended state variable. Moreover, there
exist a constant $C>0$ and a function $\ell\in L^1(0,T)$ such that
	\begin{equation}\label{eq:FG-quadratic-growth}
		|F(t,\mathbf x,a)|\leq \ell(t)+C\bigl(|\mathbf x|^2+|a|^2\bigr),\qquad |G(\mathbf x)|\leq C\bigl(1+|\mathbf x|^2\bigr).
	\end{equation}	
    Define
    \[
    J^\Phi(t,\mathbf x;u):=\mathbb E\left[\int_t^T F(s,\mathbf X_s^{t,\mathbf x,u},u_s)\,ds+G(\mathbf X_T^{t,\mathbf x,u})\right],
    \]
where \(\mathbf X^{t,\mathbf x,u}\) denotes the extended state starting from \(\mathbf x\) at time \(t\). By \eqref{eq:FG-quadratic-growth} and the standard
estimate for linear SDEs, \(J^\Phi(t,\mathbf x;u)\) is well defined for every \(u\in\mathcal U^{(N)}\). Define the value function
    \begin{equation}\label{eq:potential-value-function}
	    \mathcal V(t,\mathbf x):=\inf_{u\in\mathcal U^{(N)}}J^\Phi(t,\mathbf x;u).
    \end{equation}
    Write $\mathbf x_0:=(x,0,\ldots,0)$. We have,
    \[
    \inf_{u\in\mathcal U^{(N)}}\Phi^{t,x}(u)=\mathcal V(t,\mathbf x_0).
    \]

    The HJB equation associated with \eqref{eq:potential-value-function} is
    \begin{equation}\label{eq:potential-HJB}
	\left\{
	\begin{aligned}
		&\partial_t V(t,\mathbf x)+\inf_{a\in\mathbb A}\left\{\mathcal L^aV(t,\mathbf x)+F(t,\mathbf x,a)\right\}=0,\qquad (t,\mathbf x)\in[0,T)\times\mathbb R^{n(N+1)},\\
		&V(T,\mathbf x)={}G(\mathbf x).
	\end{aligned}
	\right.
    \end{equation}

    We have transformed the minimization of the \(\alpha\)-potential function into a finite-dimensional stochastic control problem with extended state
\eqref{eq:extended-state-SDE}. The following result is a standard consequence of the classical verification theorem for finite-dimensional stochastic control. See, for instance, \cite[Section~3.5]{Pham2009}.
	
\begin{mythm}[Verification theorem]\label{thm:verification}
	Assume that there exists $V\in C^{1,2}([0,T]\times\mathbb R^{n(N+1)})$ with at most quadratic growth such that \(V\) satisfies the HJB equation \eqref{eq:potential-HJB}. 
    Assume further that there exists a measurable map $\widehat a:[0,T]\times\mathbb R^{n(N+1)}\to \mathbb A$ such that
	\begin{equation*}
		\widehat a(t,\mathbf x)\in\arg\min_{a\in\mathbb A}\left\{\mathcal L^aV(t,\mathbf x)+F(t,\mathbf x,a)\right\},
	\end{equation*}
    and that, for every \((t,\mathbf x)\), the extended state equation
	\begin{equation}\label{eq:closed-loop-extended-state}
		\left\{
		\begin{aligned}
			d\widehat{\mathbf X}_s={}&\mu(s,\widehat{\mathbf X}_s,\widehat a(s,\widehat{\mathbf X}_s))\,ds
            +\sum_{h=1}^{d_W}\Lambda^h(s,\widehat{\mathbf X}_s,\widehat a(s,\widehat{\mathbf X}_s))\,dW_s^h,\qquad s\in[t,T],\\
			\widehat{\mathbf X}_t={}&\mathbf x,
		\end{aligned}
		\right.
	\end{equation}
	admits a square-integrable strong solution and the corresponding control $\widehat u_s:=\widehat a(s,\widehat{\mathbf X}_s)$ belongs to \(\mathcal U^{(N)}\). Then
	\[
	V(t,\mathbf x)=\mathcal V(t,\mathbf x).
	\]
	In particular, for \(\mathbf x_0=(x,0,\ldots,0)\), let \(\widehat{\mathbf X}^{t,\mathbf x_0}\) be the solution of \eqref{eq:closed-loop-extended-state} starting from \(\mathbf x_0\) at time \(t\), and set
	\[
	\widehat u_s:=\widehat a(s,\widehat{\mathbf X}_s^{t,\mathbf x_0}),\qquad s\in[t,T].
	\]
	Then \(\widehat u\) minimizes the open-loop \(\alpha\)-potential function:
	\begin{equation*}
		\Phi^{t,x}(\widehat u)=\inf_{u\in\mathcal U^{(N)}}\Phi^{t,x}(u).
	\end{equation*}
	If the assumptions of Proposition~\ref{pro:open-loop-LQ-alpha} hold on the same admissible control class, then \(\widehat u\) is an open-loop \(\alpha\)-Nash equilibrium of the original game, with \(\alpha\) bounded by \eqref{eq:open-loop-alpha-estimate}.
\end{mythm}

    We now rewrite the reduced control problem in the standard LQ form. Using the notation introduced above, the extended state dynamics can be written as
    \begin{equation*}
	\left\{
	\begin{aligned}
		&d\mathbf X_s={}\bigl(\mathcal A_s\mathbf X_s+\mathcal B_su_s+\mathfrak b_s\bigr)ds
        +\sum_{h=1}^{d_W}\bigl(\mathcal C_s^h\mathbf X_s+\mathcal D_s^hu_s+\mathfrak\sigma_s^h\bigr)dW_s^h,\\
		&\mathbf X_t={}\mathbf x,
	\end{aligned}
	\right.
    \end{equation*}
where
    \[
    \mathcal A_t\equiv\operatorname{diag}\bigl(A_t,\ldots,A_t\bigr),
    \]
with \(N+1\) diagonal blocks, and
    \[
    \mathcal B_t\equiv\begin{pmatrix}
	    B_1(t)&B_2(t)&\cdots&B_N(t)\\
	    B_1(t)&0&\cdots&0\\
        0&B_2(t)&\cdots&0\\
	    \vdots&\vdots&\ddots&\vdots\\
	    0&0&\cdots&B_N(t)
    \end{pmatrix},\qquad\mathfrak b_t\equiv
    \begin{pmatrix}
	    b_t\\
	    0\\
	    \vdots\\
	    0
    \end{pmatrix}.
    \]
    For each \(h=1,\ldots,d_W\),
    \[
    \mathcal C_t^h\equiv\operatorname{diag}\bigl(C_t^h,\ldots,C_t^h\bigr),
    \]
with \(N+1\) diagonal blocks,
    \[
    \mathcal D_t^h\equiv
    \begin{pmatrix}
	    D_1^h(t)&D_2^h(t)&\cdots&D_N^h(t)\\
	    D_1^h(t)&0&\cdots&0\\
	    0&D_2^h(t)&\cdots&0\\
        \vdots&\vdots&\ddots&\vdots\\
	    0&0&\cdots&D_N^h(t)
    \end{pmatrix},\qquad\mathfrak\sigma_t^h\equiv
    \begin{pmatrix}
	    \sigma_t^h\\
	    0\\
	    \vdots\\
    	0
    \end{pmatrix}.
    \]

    The running cost has the quadratic representation
    \begin{equation*}
	    F(t,\mathbf x,u)=\mathbf x^\top\mathcal Q_t\mathbf x+2\mathbf x^\top\mathcal S_tu+u^\top\mathcal R_tu+2\mathfrak q_t^\top\mathbf x+2\mathfrak\rho_t^\top u,
    \end{equation*}
where
    \[
    \mathcal Q_t=\frac14\begin{pmatrix}
	    0&2Q_1(t)&2Q_2(t)&\cdots&2Q_N(t)\\[1mm]
	    2Q_1(t)&-2Q_1(t)&-\bigl(Q_1(t)+Q_2(t)\bigr)&\cdots&-\bigl(Q_1(t)+Q_N(t)\bigr)\\[1mm]
	    2Q_2(t)&-\bigl(Q_2(t)+Q_1(t)\bigr)&-2Q_2(t)&\cdots&-\bigl(Q_2(t)+Q_N(t)\bigr)\\
	    \vdots&\vdots&\vdots&\ddots&\vdots\\[1mm]
    	2Q_N(t)&-\bigl(Q_N(t)+Q_1(t)\bigr)&-\bigl(Q_N(t)+Q_2(t)\bigr)&\cdots&-2Q_N(t)
    \end{pmatrix},
    \]
    \[
    \mathcal S_t=\frac14\begin{pmatrix}
	    2S_1^1(t) & 2S_2^2(t) & \cdots & 2S_N^N(t)\\
	    0 & S_1^2(t)-S_2^2(t) & \cdots & S_1^N(t)-S_N^N(t)\\
	    S_2^1(t)-S_1^1(t) & 0 & \cdots & S_2^N(t)-S_N^N(t)\\
	    \vdots & \vdots & \ddots & \vdots\\
	    S_N^1(t)-S_1^1(t) & S_N^2(t)-S_2^2(t) & \cdots & 0
    \end{pmatrix},
    \]
    \[
    \mathcal R_t=\frac14\begin{pmatrix}
	    2R_1^{11}(t)&R_1^{12}(t)+\bigl(R_2^{21}(t)\bigr)^\top&\cdots&R_1^{1N}(t)+\bigl(R_N^{N1}(t)\bigr)^\top\\[2mm]
	    R_2^{21}(t)+\bigl(R_1^{12}(t)\bigr)^\top&2R_2^{22}(t)&\cdots&R_2^{2N}(t)+\bigl(R_N^{N2}(t)\bigr)^\top\\
        \vdots&\vdots&\ddots&\vdots\\[2mm]
	    R_N^{N1}(t)+\bigl(R_1^{1N}(t)\bigr)^\top&R_N^{N2}(t)+\bigl(R_2^{2N}(t)\bigr)^\top&\cdots&2R_N^{NN}(t)
    \end{pmatrix},
    \]
    \[
    \mathfrak q_t=\frac12
    \begin{pmatrix}
	    0\\
    	q_1(t)\\
	    q_2(t)\\
	    \vdots\\
	    q_N(t)
    \end{pmatrix},\qquad\mathfrak\rho_t=\frac12
    \begin{pmatrix}
	    \rho_1^1(t)\\
	    \rho_2^2(t)\\
	    \vdots\\
	    \rho_N^N(t)
    \end{pmatrix}.
    \]

    Similarly, the terminal cost is
    \begin{equation*}
	    G(\mathbf x)=\mathbf x^\top\mathcal G\mathbf x+2\mathfrak g^\top\mathbf x,
    \end{equation*}
where
    \[
    \mathcal G=\frac14
    \begin{pmatrix}
	    0&2G_1&2G_2&\cdots&2G_N\\[1mm]
	    2G_1&-2G_1&-(G_1+G_2)&\cdots&-(G_1+G_N)\\[1mm]
	    2G_2&-(G_2+G_1)&-2G_2&\cdots&-(G_2+G_N)\\
	    \vdots&\vdots&\vdots&\ddots&\vdots\\[1mm]
	    2G_N&-(G_N+G_1)&-(G_N+G_2)&\cdots&-2G_N
    \end{pmatrix},\qquad\mathfrak g=\frac12
    \begin{pmatrix}
	    0\\
	    g_1\\
	    g_2\\
	    \vdots\\
	    g_N
    \end{pmatrix}.
    \]

  The reduced problem is a standard stochastic LQ control problem. Define
    \begin{equation*}
	\begin{aligned}
		H_t&:=\mathcal R_t+\sum_{h=1}^{d_W}(\mathcal D_t^h)^\top P_t\mathcal D_t^h,\\
		\Theta_t&:=\mathcal B_t^\top P_t+\sum_{h=1}^{d_W}(\mathcal D_t^h)^\top P_t\mathcal C_t^h+\mathcal S_t^\top,\\
		\vartheta_t&:=\mathcal B_t^\top p_t+\sum_{h=1}^{d_W}(\mathcal D_t^h)^\top P_t\mathfrak\sigma_t^h+\mathfrak\rho_t,
	\end{aligned}
    \end{equation*}
where \(P,p,\eta\) solve the following Riccati system:
    \begin{equation}\label{eq:potential-Riccati-P}
	\left\{
	\begin{aligned}
		\dot P_t&+P_t\mathcal A_t+\mathcal A_t^\top P_t+\sum_{h=1}^{d_W}(\mathcal C_t^h)^\top P_t\mathcal C_t^h+\mathcal Q_t-\Theta_t^\top H_t^{-1}\Theta_t=0,\\
		P_T&=\mathcal G,
	\end{aligned}
	\right.
    \end{equation}
    \begin{equation}\label{eq:potential-Riccati-p}
	\left\{
	\begin{aligned}
		\dot p_t&+\mathcal A_t^\top p_t+P_t\mathfrak b_t+\sum_{h=1}^{d_W}(\mathcal C_t^h)^\top P_t\mathfrak\sigma_t^h+\mathfrak q_t-\Theta_t^\top H_t^{-1}\vartheta_t=0,\\
		p_T&=\mathfrak g,
	\end{aligned}
	\right.
    \end{equation}
    \begin{equation}\label{eq:potential-Riccati-eta}
	\left\{
	\begin{aligned}
		\dot\eta_t&+2\mathfrak b_t^\top p_t+\sum_{h=1}^{d_W}(\mathfrak\sigma_t^h)^\top P_t\mathfrak\sigma_t^h-\vartheta_t^\top H_t^{-1}\vartheta_t=0,\\
		\eta_T&=0.
	\end{aligned}
	\right.
    \end{equation}

    Suppose that \eqref{eq:potential-Riccati-P}--\eqref{eq:potential-Riccati-eta} admit a solution on \([t,T]\) and that there exists \(\delta>0\) such that 
$H_s\geq\delta I_k,\forall s\in[t,T]$. This assumption is satisfied, for instance, under the standard conditions
    \[
    \mathcal G\ge0,\qquad\mathcal R_s\ge\delta I_k,\qquad
        \begin{pmatrix}
	        \mathcal Q_s & \mathcal S_s\\
	        \mathcal S_s^\top & \mathcal R_s
        \end{pmatrix}\ge0,\qquad\text{for a.e. }s\in[t,T],
    \]

    Then, by completing the square in the control variable, the infimum in the HJB equation is attained at
    \[
    \widehat a(s,\mathbf x)=-H_s^{-1}\big(\Theta_s\mathbf x+\vartheta_s\big),\qquad(s,\mathbf x)\in[t,T]\times\mathbb R^{n(N+1)}.
    \]
    Recall that $\mathbf x_0=(x,0,\ldots,0)\in\mathbb R^{n(N+1)}$. Let \(\widehat{\mathbf X}^{t,\mathbf x_0}\) be the extended state process generated by the feedback
\(\widehat a\), and define
    \[
    \widehat u_s:=\widehat a(s,\widehat{\mathbf X}^{t,\mathbf x_0}_s)=-H_s^{-1}\bigl(\Theta_s\widehat{\mathbf X}^{t,\mathbf x_0}_s+\vartheta_s\bigr),\qquad s\in[t,T].
    \]
    If \(\widehat u\in\mathcal U^{(N)}\), then, by Theorem~\ref{thm:verification},
    \[
    \Phi^{t,x}(\widehat u)=\inf_{u\in\mathcal U^{(N)}}\Phi^{t,x}(u).
    \]
    Consequently, under the assumptions of Proposition~\ref{pro:open-loop-LQ-alpha}, \(\widehat u\) is an open-loop \(\alpha\)-Nash equilibrium.

\section{Application: A network LQ game}\label{sec:application}

    We now apply the preceding open-loop potential-minimization method to the network LQ model in \cite[Section~6]{GuoLiZhang2025}. The aim is to compare the reduced 
finite-dimensional LQ control problem obtained from our method with the ODE system characterizing the $\alpha$-potential minimizer in \cite[Theorem~6.1]{GuoLiZhang2025}.
We prove that the two formulations  yield the same feedback representation.
	
\subsection{The recalled network LQ model}
	
	Throughout this section, let $(\Omega,\mathcal F,\mathbb F,\mathbb P)$ be a complete filtered probability space and $W\equiv(W_1,\ldots,W_N)^\top$ be an $N$-dimensional
Brownian motion, $\mathbb F=(\mathcal F_s)_{s\in[t,T]}$ is the completed natural filtration generated by $W$.
		
	Let $X=(X_1,\ldots,X_N)^\top$ and $u=(u_1,\ldots,u_N)^\top$. For each \(i\in\mathcal I_N\), consider
	\begin{equation}\label{eq:app-state}
		\left\{
		\begin{aligned}
			dX_{i,s}^{u}={}&\bigl(a_i(s)X_{i,s}^{u}+u_{i,s}\bigr)ds+\sigma_i(s)dW_{i,s},\qquad s\in[t,T],\\
			X_{i,t}^{u}={}&x_i,
		\end{aligned}
		\right.
	\end{equation}
	where $a_i,\sigma_i\in C([0,T];\mathbb R)$. Player $i$ has cost
	\begin{equation}\label{eq:app-cost}
		J_i^{t,x}(u)=\mathbb E\Bigg[\int_t^T\Bigg(u_{i,s}^2+\frac1N\sum_{j=1}^Nq_{ij}(X_{i,s}^u-X_{j,s}^u)^2\Bigg)ds+\gamma_i(X_{i,T}^u-d_i)^2\Bigg],
	\end{equation}
where $q_{ij}\ge0$, $\gamma_i\ge0$, and $d_i\in\mathbb R$. The admissible control set is
	\begin{equation*}
		\mathscr A_i=\left\{u_i\in\mathcal H^2([t,T];\mathbb R):\|u_i\|_{\mathcal H^2([t,T];\mathbb R)}\le L\right\},
	\end{equation*}
where $L>0$ is sufficiently large. Set $\mathscr A^{(N)}:=\prod_{i=1}^N\mathscr A_i$.
	
	For a given control $u$, define the variational process $Y^u=(Y_1^u,\ldots,Y_N^u)^\top$ by
	\begin{equation*}
		\left\{
		\begin{aligned}
			dY_{i,s}^{u}={}&\bigl(a_i(s)Y_{i,s}^{u}+u_{i,s}\bigr)ds,\qquad s\in[t,T],\\
			Y_{i,t}^{u}={}&0.
		\end{aligned}
		\right.
	\end{equation*}
	By linearity,
	\begin{equation}\label{eq:app-Xru-Xu-Y}
		X_s^{ru}=X_s^u-(1-r)Y_s^u,\qquad r\in[0,1].
	\end{equation}
	
	The following result is taken from \cite[Theorem~6.1]{GuoLiZhang2025}. For the game \eqref{eq:app-state}--\eqref{eq:app-cost},
	\begin{equation}\label{eq:app-alpha-potential}
		\Phi^{t,x}(u)=\int_0^1\sum_{i=1}^N\frac{\delta J_i^{t,x}}{\delta u_i}(ru;u_i)\,dr
	\end{equation}
is an $\alpha_N$-potential function, and
	\begin{equation*}
		\alpha_N\le C\frac1N\max_{1\le i\le N}\sum_{j\ne i}|q_{ji}-q_{ij}|,
	\end{equation*}
where $C$ depends only on the model bounds and the time horizon.
	
	Let \(\tau\) be a uniform random variable on \([0,1]\), independent of \(\mathcal F_T\), and define
	\begin{equation*}
		\mathbb X_s^{\tau,u}:=\begin{pmatrix}X_s^{\tau u}\\Y_s^u\end{pmatrix}\in\mathbb R^{2N}.
	\end{equation*}
	Set
	\[
	A_s^N=\operatorname{diag}(a_1(s),\ldots,a_N(s)),\qquad\Sigma_s^N=\operatorname{diag}(\sigma_1(s),\ldots,\sigma_N(s)),
	\]
	\[
	\mathcal A_s=\begin{pmatrix}A_s^N&0\\0&A_s^N\end{pmatrix},\qquad\bar\Sigma_s=\begin{pmatrix}\Sigma_s^N\\0\end{pmatrix},\qquad
    \mathcal B=\begin{pmatrix}I_N\\I_N\end{pmatrix}.
	\]
	For $r\in[0,1]$, define
	\[
	D_r=\begin{pmatrix}I_N&-(1-r)I_N\\0&I_N\end{pmatrix},\qquad I_r=D_r\mathcal B=\begin{pmatrix}rI_N\\I_N\end{pmatrix}.
	\]
	Then $\mathbb X^{\tau,u}$ satisfies
	\begin{equation*}
		d\mathbb X_s^{\tau,u}=\bigl(\mathcal A_s\mathbb X_s^{\tau,u}+I_\tau u_s\bigr)ds+\bar\Sigma_s dW_s,\qquad\mathbb X_t^{\tau,u}=\begin{pmatrix}x\\0\end{pmatrix}.
	\end{equation*}
	Let $S:=\mathbb R^{2N}\times[0,1]$. For a given control $u\in\mathscr A^{(N)}$, define the conditional law flow $\mu^{\tau,u}=(\mu_s^{\tau,u})_{s\in[t,T]}$ on $S$ by
	\begin{equation*}
		\mu_s^{\tau,u}:=\mathscr L\bigl(\mathbb X_s^{\tau,u},\tau\mid\mathcal F_s\bigr),\qquad s\in[t,T].
	\end{equation*}
    The \(\alpha\)-potential function can then be written as
	\begin{equation}\label{eq:app-measure-potential}
		\Phi^{t,x}(u)=\mathbb E\left[\int_t^T\int_S\left(\mathbf x^\top Q\mathbf x+2r u_s^\top u_s\right)\,d\mu_s^{\tau,u}(\mathbf x,r)\,ds
        +\int_S\left(\mathbf x^\top \bar Q\mathbf x+2(p_0)^\top\mathbf x\right)\,d\mu_T^{\tau,u}(\mathbf x,r)\right].
	\end{equation}
	Here $Q,\bar Q\in\mathbb S^{2N}$ and $p_0\in\mathbb R^{2N}$ are defined as follows. Let $\widetilde Q\in\mathbb R^{N\times N}$ be given by
	\[
	\widetilde Q_{ii}=\frac1N\sum_{k\ne i}q_{ik},\qquad\widetilde Q_{ij}=-\frac{q_{ij}}{N},\quad i\ne j.
	\]
	Then
	\begin{equation*}
		Q=\begin{pmatrix}0&\widetilde Q^\top\\ \widetilde Q&0\end{pmatrix},\qquad\bar Q=\begin{pmatrix}0&\Gamma\\\Gamma&0\end{pmatrix},\qquad
		p_0=-\begin{pmatrix}0\\\Gamma d\end{pmatrix},
	\end{equation*}
	with $\Gamma=\operatorname{diag}(\gamma_1,\ldots,\gamma_N)$ and $d=(d_1,\ldots,d_N)^\top$.
	
	For \(u\in\mathscr A^{(N)}\), define $F^u=(F_s^u)_{s\in[t,T]}$ by
	\begin{equation*}
		F_s^u:=
		\begin{pmatrix}
			\displaystyle\int_S \mathbf x\,d\mu_s^{\tau,u}(\mathbf x,r)\\[1mm]
			\displaystyle\int_S r\mathbf x\,d\mu_s^{\tau,u}(\mathbf x,r)
		\end{pmatrix}=
		\begin{pmatrix}
			\mathbb E[\mathbb X_s^{\tau,u}\mid\mathcal F_s]\\
			\mathbb E[\tau\mathbb X_s^{\tau,u}\mid\mathcal F_s]
		\end{pmatrix}\in\mathbb R^{4N}.
	\end{equation*}
	\[
	\mathbb A_s=\begin{pmatrix}\mathcal A_s&0\\0&\mathcal A_s\end{pmatrix},\qquad\widetilde I=\left(\frac12I_N,I_N,\frac13I_N,\frac12I_N\right),\qquad
	\mathcal J=\begin{pmatrix}I_{2N}\\ \frac12 I_{2N}\end{pmatrix}.
	\]
	Suppose $M_0,M_1,M_2,M_3$ solve
	\begin{equation}\label{eq:app-M0}
		\dot M_0+\mathcal A^\top M_0+M_0\mathcal A+Q=0,\qquad M_0(T)=\bar Q,
	\end{equation}
	\begin{equation}\label{eq:app-M1}
		\dot M_1+\mathbb A^\top M_1+M_1\mathbb A-K^\top K=0,\qquad M_1(T)=0,
	\end{equation}
	\begin{equation}\label{eq:app-M2}
		\dot M_2+\mathbb A^\top M_2-K^\top\widetilde I M_2=0,\qquad M_2(T)=\begin{pmatrix}p_0\\0\end{pmatrix},
	\end{equation}
and
	\begin{equation}\label{eq:app-M3}
		\dot M_3+\operatorname{tr}\left(\bar\Sigma\bar\Sigma^\top\left[M_0+\mathcal J^\top M_1\mathcal J\right]\right)-(\widetilde I M_2)^\top(\widetilde I M_2)=0,
		\qquad M_3(T)=0,
	\end{equation}
where
	\begin{equation*}
		K=\bigl((0,I_N)M_0,(I_N,0)M_0\bigr)+\widetilde I M_1.
	\end{equation*}
	If there exists a control $\widehat u^{\rm ref}\in\mathscr A^{(N)}$ such that, with $\widehat F_s:=F_s^{\widehat u^{\rm ref}}$, one has
	\begin{equation}\label{eq:app-recalled-feedback}
		\widehat u_s^{\rm ref}=-K_s\widehat F_s-\widetilde I M_2(s),\qquad s\in[t,T],
	\end{equation}
then $\widehat u^{\rm ref}$ is an $\alpha_N$-open-loop Nash equilibrium.
	
\subsection{The reduced Riccati system and feedback comparison}
	
	The recalled construction above still involves the random variable $\tau$. In the present LQ model, we can eliminate $\tau$ from the state by using the linear identity \eqref{eq:app-Xru-Xu-Y}. Define the reduced state
	\begin{equation*}
		\mathbf X_s^u:=\begin{pmatrix}X_s^u\\Y_s^u\end{pmatrix}\in\mathbb R^{2N}.
	\end{equation*}
	Then
	\begin{equation}\label{eq:app-bold-reduced-relation}
		\mathbb X_s^{r,u}=D_r\mathbf X_s^u,\qquad r\in[0,1].
	\end{equation}
	Moreover, $\mathbf X^u$ satisfies
	\begin{equation*}
		\left\{
		\begin{aligned}
			d\mathbf X_s^u&=(\mathcal A_s\mathbf X_s^u+\mathcal Bu_s)ds+\bar\Sigma_s dW_s,\qquad s\in[t,T],\\
			\mathbf X_t^u&=\begin{pmatrix}x\\0\end{pmatrix}.
		\end{aligned}
		\right.
	\end{equation*}
	Substituting \eqref{eq:app-Xru-Xu-Y} into \eqref{eq:app-alpha-potential}, equivalently into
	\eqref{eq:app-measure-potential}, yields
	\begin{equation}\label{eq:app-reduced-Phi}
		\Phi^{t,x}(u)=\mathbb E\left[\int_t^T\left(u_s^\top u_s+(\mathbf X_s^u)^\top\mathcal Q\mathbf X_s^u\right)ds
		+(\mathbf X_T^u)^\top\mathcal G\mathbf X_T^u+2\mathfrak g^\top\mathbf X_T^u\right],
	\end{equation}
	where
	\begin{equation*}
		\mathcal Q=\int_0^1D_r^\top QD_r\,dr,\qquad\mathcal G=\int_0^1D_r^\top\bar QD_r\,dr,\qquad\mathfrak g=\int_0^1D_r^\top p_0\,dr.
	\end{equation*}
	Explicitly,
	\[
	\mathcal Q=
	\begin{pmatrix}
		0&\widetilde Q^\top\\[1mm]
		\widetilde Q&-\frac12(\widetilde Q+\widetilde Q^\top)
	\end{pmatrix},\qquad\mathcal G=
	\begin{pmatrix}
		0&\Gamma\\
		\Gamma&-\Gamma
	\end{pmatrix},\qquad\mathfrak g=-\begin{pmatrix}0\\\Gamma d\end{pmatrix}.
	\]
	
	For \(\mathbf x\in\mathbb R^{2N}\), let \(\mathbf X^{t,\mathbf x,u}\) be the solution of
	\begin{equation*}
		\left\{
		\begin{aligned}
			d\mathbf X_s^{t,\mathbf x,u}&=(\mathcal A_s\mathbf X_s^{t,\mathbf x,u}+\mathcal Bu_s)ds+\bar\Sigma_s dW_s,\qquad s\in[t,T],\\
			\mathbf X_t^{t,\mathbf x,u}&=\mathbf x .
		\end{aligned}
		\right.
	\end{equation*}
	Define the value function of the reduced LQ problem by
	\begin{equation*}
		\begin{aligned}
			\mathcal V(t,\mathbf x):=\inf_{u\in\mathscr A^{(N)}}\mathbb E\bigg[&
			\int_t^T\left(u_s^\top u_s+(\mathbf X_s^{t,\mathbf x,u})^\top\mathcal Q\mathbf X_s^{t,\mathbf x,u}\right)ds
			+(\mathbf X_T^{t,\mathbf x,u})^\top\mathcal G\mathbf X_T^{t,\mathbf x,u}+2\mathfrak g^\top\mathbf X_T^{t,\mathbf x,u}\bigg].
		\end{aligned}
	\end{equation*}
	In particular, the cost in \eqref{eq:app-reduced-Phi} corresponds to the initial condition \(\mathbf x=(x,0)^\top\). 
	
	Completing the square in the Hamiltonian gives the minimizer
	\[
	\widehat a(t,\mathbf x):=-\mathcal B^\top(P_t\mathbf x+p_t),
	\]
	where
    \begin{equation}\label{eq:app-Riccati-P}
	\left\{
	\begin{aligned}
		&\dot P+\mathcal A^\top P+P\mathcal A+\mathcal Q-P\mathcal B\mathcal B^\top P=0,\\
		&P_T={}\mathcal G,
	\end{aligned}
	\right.
    \end{equation}
    \begin{equation}\label{eq:app-Riccati-p}
	\left\{
	\begin{aligned}
		&\dot p+\mathcal A^\top p-P\mathcal B\mathcal B^\top p=0,\\
		&p_T={}\mathfrak g,
	\end{aligned}
	\right.
    \end{equation}
    \begin{equation}\label{eq:app-Riccati-eta}
	\left\{
	\begin{aligned}
		&\dot\eta+\operatorname{tr}(\bar\Sigma\bar\Sigma^\top P)-p^\top\mathcal B\mathcal B^\top p=0,\\
		&\eta_T={}0.
	\end{aligned}
	\right.
    \end{equation}
    Let
    \begin{equation}\label{eq:app-reduced-feedback}
	\widehat u_s=-\mathcal B^\top(P_s\widehat{\mathbf X}_s+p_s).
    \end{equation}
    Here \(\widehat{\mathbf X}\) denote the solution of the state equation under the feedback control \eqref{eq:app-reduced-feedback}. If \(\widehat u\in\mathscr A^{(N)}\), 
then \(\widehat u\) is optimal over \(\mathscr A^{(N)}\).

	We now prove that the reduced Riccati system \eqref{eq:app-Riccati-P}--\eqref{eq:app-Riccati-eta} can be obtained from the Riccati system \eqref{eq:app-M0}--\eqref{eq:app-M3}. By \eqref{eq:app-bold-reduced-relation} and the independence of $\tau$ from $\mathcal F_s$, $F^u$ satisfies
	\begin{equation}\label{eq:app-F-reduced-relation}
		F_s^u=\begin{pmatrix}
			\mathbb E[D_\tau]\\[1mm]
			\mathbb E[\tau D_\tau]
		\end{pmatrix}\mathbf X_s^u=\mathcal T\mathbf X_s^u,
	\end{equation}
	where
	\begin{equation*}
		\mathcal T=\begin{pmatrix}
			\displaystyle\int_0^1D_r\,dr\\[1mm]
			\displaystyle\int_0^1rD_r\,dr
		\end{pmatrix}=
		\begin{pmatrix}
			I_N&-\frac12I_N\\
			0&I_N\\
			\frac12I_N&-\frac16I_N\\
			0&\frac12I_N
		\end{pmatrix}.
	\end{equation*}
    We shall use the following identities, all of which follow directly from the definitions:
    \begin{equation}\label{eq:app-algebra-identities}
        D_r\mathcal A_s=\mathcal A_sD_r,\qquad D_r\mathcal B=I_r,\qquad D_r\bar\Sigma_s=\bar\Sigma_s,\qquad \mathbb A_s\mathcal T=\mathcal T\mathcal A_s,\qquad
        \mathcal B^\top\mathcal T^\top=\widetilde I .
    \end{equation}
    
	\begin{mypro}\label{pro:app-Riccati-projection}
		Assume that $M_0,M_1,M_2,M_3$ solve \eqref{eq:app-M0}--\eqref{eq:app-M3}. Define
		\begin{equation*}
			\bar P_s:=\int_0^1D_r^\top M_0(s)D_r\,dr+\mathcal T^\top M_1(s)\mathcal T,
		\end{equation*}
		\begin{equation*}
			\bar p_s:=\mathcal T^\top M_2(s),\qquad\bar\eta_s:=M_3(s).
		\end{equation*}
		Then $(\bar P,\bar p,\bar\eta)$ solves \eqref{eq:app-Riccati-P}--\eqref{eq:app-Riccati-eta}. Hence, by the uniqueness of the Riccati system
		\[
		P=\bar P,\qquad p=\bar p,\qquad\eta=\bar\eta.
		\]
	\end{mypro}
	
	\begin{proof}
		Write
		\[
		\bar P=P^{(0)}+P^{(1)},\qquad P_s^{(0)}=\int_0^1D_r^\top M_0(s)D_r\,dr,\qquad P_s^{(1)}=\mathcal T^\top M_1(s)\mathcal T.
		\]
		From \eqref{eq:app-M0} and $D_r\mathcal A=\mathcal A D_r$,
		\[
		\begin{aligned}
			\dot P^{(0)}&=\int_0^1D_r^\top\dot M_0D_r\,dr \\
			&=-\int_0^1D_r^\top\mathcal A^\top M_0D_r\,dr-\int_0^1D_r^\top M_0\mathcal A D_r\,dr-\int_0^1D_r^\top QD_r\,dr \\
			&=-\mathcal A^\top P^{(0)}-P^{(0)}\mathcal A-\mathcal Q.
		\end{aligned}
		\]
		From \eqref{eq:app-M1} and $\mathbb A\mathcal T=\mathcal T\mathcal A$,
		\[
		\begin{aligned}
			\dot P^{(1)}&=\mathcal T^\top\dot M_1\mathcal T
			=-\mathcal T^\top\mathbb A^\top M_1\mathcal T-\mathcal T^\top M_1\mathbb A\mathcal T+\mathcal T^\top K^\top K\mathcal T \\
			&=-\mathcal A^\top P^{(1)}-P^{(1)}\mathcal A+(K\mathcal T)^\top(K\mathcal T).
		\end{aligned}
		\]

		We claim that
		\begin{equation}\label{eq:app-KT-BPbar}
			K\mathcal T=\mathcal B^\top\bar P.
		\end{equation}
		Indeed, for any $v\in\mathbb R^{2N}$,
		\[
		\begin{aligned}
			\bigl((0,I_N)M_0,(I_N,0)M_0\bigr)\mathcal T v&=(0,I_N)M_0\int_0^1D_r v\,dr+(I_N,0)M_0\int_0^1rD_r v\,dr \\
			&=\int_0^1\bigl((0,I_N)+r(I_N,0)\bigr)M_0D_r v\,dr.
		\end{aligned}
		\]
		Since
		\[
		(0,I_N)+r(I_N,0)=(rI_N,I_N)=(D_r\mathcal B)^\top=\mathcal B^\top D_r^\top,
		\]
		we get
		\[
		\bigl((0,I_N)M_0,(I_N,0)M_0\bigr)\mathcal T=\mathcal B^\top\int_0^1D_r^\top M_0D_r\,dr=\mathcal B^\top P^{(0)}.
		\]
		Moreover, by \eqref{eq:app-algebra-identities},
		\[
		\widetilde I M_1\mathcal T=\mathcal B^\top\mathcal T^\top M_1\mathcal T=\mathcal B^\top P^{(1)}.
		\]
		Thus \eqref{eq:app-KT-BPbar} follows. Consequently,
		\[
		(K\mathcal T)^\top(K\mathcal T)=\bar P\mathcal B\mathcal B^\top\bar P.
		\]
		Adding the equations for $P^{(0)}$ and $P^{(1)}$ gives
		\[
		\dot{\bar P}+\mathcal A^\top\bar P+\bar P\mathcal A+\mathcal Q-\bar P\mathcal B\mathcal B^\top\bar P=0.
		\]
		The terminal condition is
		\[
		\bar P_T=\int_0^1D_r^\top M_0(T)D_r\,dr+\mathcal T^\top M_1(T)\mathcal T=\int_0^1D_r^\top\bar QD_r\,dr=\mathcal G.
		\]
		Therefore $\bar P$ satisfies \eqref{eq:app-Riccati-P}. 

        Next, from \eqref{eq:app-M2},
		\[
		\begin{aligned}
			\dot{\bar p}=\mathcal T^\top\dot M_2&=-\mathcal T^\top\mathbb A^\top M_2+\mathcal T^\top K^\top\widetilde I M_2
			=-\mathcal A^\top\bar p+(K\mathcal T)^\top\widetilde I M_2.
		\end{aligned}
		\]
		Using \eqref{eq:app-KT-BPbar} and
		\[
		\widetilde I M_2=\mathcal B^\top\mathcal T^\top M_2=\mathcal B^\top\bar p,
		\]
		we obtain
		\[
		\dot{\bar p}+\mathcal A^\top\bar p-\bar P\mathcal B\mathcal B^\top\bar p=0.
		\]
		The terminal condition is
		\[
		\bar p_T=\mathcal T^\top M_2(T)=\mathcal T^\top\begin{pmatrix}p_0\\0\end{pmatrix}=\int_0^1D_r^\top p_0\,dr=\mathfrak g.
		\]
		Thus $\bar p$ satisfies \eqref{eq:app-Riccati-p}.
		
		Finally, since $\bar\eta=M_3$, we use \eqref{eq:app-M3}. By $D_r\bar\Sigma=\bar\Sigma$,
		\[
		\operatorname{tr}\left(\bar\Sigma\bar\Sigma^\top\int_0^1D_r^\top M_0D_r\,dr\right)=\operatorname{tr}(\bar\Sigma\bar\Sigma^\top M_0).
		\]
		Moreover,
		\[
		\mathcal T\bar\Sigma=\begin{pmatrix}
			\int_0^1D_r\bar\Sigma\,dr\\[1mm]
			\int_0^1rD_r\bar\Sigma\,dr
		\end{pmatrix}=
		\begin{pmatrix}\bar\Sigma\\ \frac12\bar\Sigma\end{pmatrix}=\mathcal J\bar\Sigma.
		\]
		Therefore
		\[
		\operatorname{tr}(\bar\Sigma\bar\Sigma^\top\mathcal T^\top M_1\mathcal T)=
		\operatorname{tr}((\mathcal T\bar\Sigma)(\mathcal T\bar\Sigma)^\top M_1)=
		\operatorname{tr}\left(\bar\Sigma\bar\Sigma^\top\mathcal J^\top M_1\mathcal J\right).
		\]
		Also,
		\[
		\bar p^\top\mathcal B\mathcal B^\top\bar p=(\mathcal B^\top\bar p)^\top(\mathcal B^\top\bar p)=(\widetilde I M_2)^\top(\widetilde I M_2).
		\]
		Thus \eqref{eq:app-M3} is exactly
		\[
		\dot{\bar\eta}+\operatorname{tr}(\bar\Sigma\bar\Sigma^\top\bar P)-\bar p^\top\mathcal B\mathcal B^\top\bar p=0.
		\]
		The terminal condition is $\bar\eta_T=M_3(T)=0$. Therefore $\bar\eta$ satisfies \eqref{eq:app-Riccati-eta}.
		
		Therefore $(\bar P,\bar p,\bar\eta)$ solves the Riccati system \eqref{eq:app-Riccati-P}--\eqref{eq:app-Riccati-eta}. Hence, by uniqueness of this Riccati system,
		\[
		P=\bar P,\qquad p=\bar p,\qquad\eta=\bar\eta.
		\]
		The proof is complete.
	\end{proof}
	
\begin{mycor}
	Under the assumptions of Proposition~\ref{pro:app-Riccati-projection}, consider the state-control pair \((\widehat{\mathbf X},\widehat u)\) defined by \eqref{eq:app-reduced-feedback}. If \(\widehat u\in\mathscr A^{(N)}\), then 
	\[
	F_s^{\widehat u}=\mathcal T\widehat{\mathbf X}_s,\qquad s\in[t,T].
	\]
	Moreover, \(\widehat u\) satisfies
	\[
	\widehat u_s=-K_sF_s^{\widehat u}-\widetilde I M_2(s),\qquad s\in[t,T].
	\]
	Equivalently, \(\widehat u\) satisfies the recalled feedback relation \eqref{eq:app-recalled-feedback} with
	\[
	\widehat u^{\rm ref}=\widehat u,\qquad\widehat F_s=F_s^{\widehat u}.
	\]
	Consequently, the feedback control obtained from the reduced finite-dimensional LQ problem coincides with the feedback control obtained from the recalled formulation.
\end{mycor}

\begin{proof}
	By \eqref{eq:app-F-reduced-relation}, for any admissible control \(u\),
	\[
	F_s^u=\mathcal T\mathbf X_s^u,\qquad s\in[t,T].
	\]
	Applying this identity to \(u=\widehat u\) gives
	\[
	F_s^{\widehat u}=\mathcal T\widehat{\mathbf X}_s.
	\]
	
	By Proposition~\ref{pro:app-Riccati-projection},
	\[
	P_s=\bar P_s,\qquad p_s=\bar p_s.
	\]
	Using \eqref{eq:app-KT-BPbar}, the definition \(\bar p_s=\mathcal T^\top M_2(s)\), and \eqref{eq:app-algebra-identities}, we obtain
	\[
	K_s\mathcal T=\mathcal B^\top P_s,\qquad\widetilde I M_2(s)=\mathcal B^\top p_s.
	\]
	Therefore, by \eqref{eq:app-reduced-feedback},
	\[
	\begin{aligned}
		\widehat u_s=-\mathcal B^\top(P_s\widehat{\mathbf X}_s+p_s)=-K_s\mathcal T\widehat{\mathbf X}_s-\widetilde I M_2(s)=-K_sF_s^{\widehat u}-\widetilde I M_2(s).
	\end{aligned}
	\]
	This is precisely \eqref{eq:app-recalled-feedback} with \(\widehat u^{\rm ref}=\widehat u\) and \(\widehat F_s=F_s^{\widehat u}\). The proof is complete.
\end{proof}
	
\section{Concluding remarks}
	
	In this paper, we studied stochastic LQ differential games through the $\alpha$-potential approach. In the closed-loop setting, we established the equivalence between the probabilistic and PDE representations of the first and second order linear derivatives. In the open-loop setting, we constructed the $\alpha$-potential function, derived an explicit upper bound for the parameter $\alpha$, and reduced the minimization of the $\alpha$-potential function to a finite-dimensional LQ control problem. The network LQ example further shows the consistency between the reduced finite-dimensional control problem and the characterization obtained through the conditional McKean--Vlasov approach.

    Several extensions would be worth investigating. In particular, it would be interesting to consider games under partial observation or with delays in the state dynamics. 
Another natural direction is to extend the $\alpha$-potential approach to stochastic leader–follower games.


\begin{thebibliography}{99}

    \bibitem{CachonZipkin1999}
	G. P. Cachon and P. H. Zipkin,
	\newblock Competitive and cooperative inventory policies in a two-stage supply chain,
	\newblock {\it Management Science}, 45(7), 936--953, 1999.

    \bibitem{CanalesGallego2010}
	M. Canales and J. R. Gallego,
	\newblock Potential game for joint channel and power allocation in cognitive radio networks,
	\newblock {\it Electronics Letters}, 46(24), 1632--1634, 2010.

	\bibitem{DiHuWangZhang2025}
	X. Di, A. Hu, Z. Wang and Y. Zhang,
	\newblock $\alpha$-potential games for decentralized control of connected and automated vehicles,
	\newblock {\it arXiv:2512.05712}, 2025.

    \bibitem{GuoLiMaheshwariSastryWu2026}
	X. Guo, X. Li, C. Maheshwari, S. Sastry and M. Wu,
	\newblock Markov $\alpha$-potential games,
	\newblock {\it IEEE Transactions on Automatic Control}, 71(1), 275--290, 2026.
 
    \bibitem{GuoLiZhangBSDE2025}
	X. Guo, X. Li and L. Zhang,
	\newblock BSDE approach for $\alpha$-potential stochastic differential games,
	\newblock {\it arXiv:2507.13256}, 2025.

    \bibitem{GuoLiZhang2025}
	X. Guo, X. Li and Y. Zhang,
	\newblock An $\alpha$-potential game framework for $N$-player dynamic games,
	\newblock {\it SIAM Journal on Control and Optimization}, 63(4), 2964--3005, 2025.
	
	\bibitem{GuoLiZhangJump2025}
	X. Guo, X. Li and Y. Zhang,
	\newblock Distributed games with jumps: an $\alpha$-potential game approach,
	\newblock arXiv:2508.01929, 2025.

	\bibitem{GuoWangZhang2026}
	X. Guo, M. Wang and Y. Zhang,
	\newblock Limit theory for $N$-player $\alpha$-potential games,
	\newblock {\it arXiv:2606.09815}, 2026.

    \bibitem{GuoZhang2025}
	X. Guo and Y. Zhang,
	\newblock Towards an analytical framework for dynamic potential games,
	\newblock {\it SIAM Journal on Control and Optimization}, 63(2), 1213--1242, 2025.

	\bibitem{Pham2009}
	H. Pham,
	\newblock {\it Continuous-Time Stochastic Control and Optimization with Financial Applications},
	\newblock Springer, 2009.

	\bibitem{PlankZhang2026}
	P. Plank and Y. Zhang,
	\newblock Learning distributed equilibria in linear-quadratic stochastic differential games: an $\alpha$-potential approach,
	\newblock {\it arXiv:2602.16555}, 2026.

	\bibitem{MondererShapley1996}
	D. Monderer and L. S. Shapley,
	\newblock Potential games,
	\newblock {\it Games and Economic Behavior}, 14(1), 124--143, 1996.
	
	\bibitem{Nash1950}
	J. F. Nash, Jr.,
	\newblock Equilibrium points in \(n\)-person games,
	\newblock {\it Proceedings of the National Academy of Sciences of the United States of America}, 36(1), 48--49, 1950.
	
	\bibitem{TaoFeijooLorenzo2024}
	S. Tao and A. E. Feij\'oo-Lorenzo,
	\newblock Multi-objective optimization of clustered wind farms based on potential game approach,
	\newblock {\it Ocean Engineering}, 300, 117291, 2024.
	
    \bibitem{VonNeumannMorgenstern1944}
	J. von Neumann and O. Morgenstern,
	\newblock {\it Theory of Games and Economic Behavior},
	\newblock Princeton University Press, 1944.

	\bibitem{Yamamoto2015}
	K. Yamamoto,
	\newblock A comprehensive survey of potential game approaches to wireless networks,
	\newblock {\it IEICE Transactions on Communications}, E98-B(9), 1804--1823, 2015.

\end{thebibliography}
\end{document}